\documentclass[a4paper]{amsart}
\usepackage{hyperref}
\usepackage{txfonts, amsmath,amstext,amsthm,amscd,amsopn,verbatim,amssymb, amsfonts}
\usepackage{fullpage}
\usepackage{todonotes}
\usepackage{mathtools}

\usepackage[bbgreekl]{mathbbol}
\usepackage{enumerate}

\usepackage{float}
\restylefloat{table}

\usepackage{tikz}
\usepackage{tikz-cd}
\usetikzlibrary{matrix}
\usetikzlibrary{shapes}
\usetikzlibrary{arrows}
\usetikzlibrary{calc,3d}
\usetikzlibrary{decorations,decorations.pathmorphing,decorations.pathreplacing,decorations.markings}
\usetikzlibrary{through}

\tikzset{snakeit/.style={decorate, decoration={snake, amplitude=.2mm,segment length=1mm}}}

\tikzset{ext/.style={circle, draw,inner sep=1pt}, int/.style={circle,draw,fill,inner sep=2pt},nil/.style={inner sep=1pt}}
\tikzset{cy/.style={circle,draw,fill,inner sep=2pt},scy/.style={circle,draw,inner sep=2pt},scyx/.style={draw,cross out,inner sep=2pt},scyt/.style={draw,regular polygon,regular polygon sides=3,inner sep=0.95pt}}
\tikzset{exte/.style={circle, draw,inner sep=3pt},inte/.style={circle,draw,fill,inner sep=3pt}}
\tikzset{diagram/.style={matrix of math nodes, row sep=3em, column sep=2.5em, text height=1.5ex, text depth=0.25ex}}
\tikzset{diagram2/.style={matrix of math nodes, row sep=0.5em, column sep=0.5em, text height=1.5ex, text depth=0.25ex}}
\tikzset{rowcolsep/.style={column sep=.2cm, row sep=.1cm}}

\tikzset{
  crossed/.style={
    decoration={markings,mark=at position .5 with {\arrow{|}}},
    postaction={decorate},
    shorten >=0.4pt}}

\tikzset{every picture/.style={baseline=-.65ex} }
\tikzset{every loop/.style={draw} }

\theoremstyle{plain}
  \newtheorem{thm}{Theorem}
  \newtheorem{defi}[thm]{Definition}
  \newtheorem{prop}[thm]{Proposition}
  
  \newtheorem{cor}[thm]{Corollary}
  
  \newtheorem{lemma}[thm]{Lemma}
\theoremstyle{definition}
  \newtheorem{ex}{Example}
  \newtheorem{rem}{Remark}

\newcommand{\Hom}{\mathrm{Hom}}
\newcommand{\HHom}{\mathbb{H}\mathrm{om}}

\newcommand{\R}{{\mathbb{R}}}

\newcommand{\N}{{\mathbb{N}}}

\newcommand{\SSeq}{S\Mod}
\newcommand{\Free}{\mathit{Free}}
\DeclareMathOperator{\colim}{colim}
\newcommand{\mF}{\mathcal{F}}

\newcommand{\Def}{\mathrm{Def}}

\newcommand{\Lie}{\mathsf{Lie}}

\newcommand{\pop}{\mathsf{p}}

\newcommand{\Com}{\mathsf{Com}}

\newcommand{\FM}{\mathsf{FM}}

\newcommand{\bpm}{\begin{pmatrix}}
\newcommand{\epm}{\end{pmatrix}}

\newcommand{\hotimes}{\mathbin{\hat\otimes}}

\DeclareMathOperator{\End}{End}
\DeclareMathOperator{\sgn}{sgn}
\DeclareMathOperator{\gr}{gr}

\newcommand{\MOp}{\mathsf{M}}
\newcommand{\UOp}{\mathsf{U}}
\newcommand{\QOp}{\mathsf{Q}}
\newcommand{\AOp}{\mathsf{A}}
\newcommand{\BOp}{\mathsf{B}}
\newcommand{\NOp}{\mathsf{N}}
\newcommand{\ROp}{\mathsf{R}}
\newcommand{\mop}{\mathsf{m}}
\newcommand{\qop}{\mathsf{q}}
\newcommand{\FreeP}{\mathcal F} 

\newcommand{\Q}{\mathbb{Q}}

\newcommand{\cF}{\mathcal{F}}
\newcommand{\Map}{\mathrm{Map}}

\DeclareMathOperator{\Mor}{Mor}

\newcommand{\fg}{\mathfrak{g}}
\newcommand{\fh}{\mathfrak{h}}

\newcommand{\GX}{\mathsf{GX}}
\newcommand{\MC}{\mathrm{MC}}

\usepackage{tikz-cd}
\tikzset{%
    symbol/.style={%
        draw=none,
        every to/.append style={%
            edge node={node [sloped, allow upside down, auto=false]{$#1$}}}
    }
}

\newcommand{\Mod}{{\text{-}\mathcal{M}od}}

\newcommand{\Env}{\mathrm{Env}}

\newcommand{\Fun}{\mathrm{Fun}}

\newcommand{\cC}{\mathcal{C}}
\newcommand{\cD}{\mathcal{D}}

\newcommand{\pois}{\mathsf{pois}}
\newcommand{\Pois}{\mathsf{Pois}}
\newcommand{\uifrob}{\mathsf{uifrob}}
\newcommand{\uIFrob}{\mathsf{uIFrob}}
\newcommand{\ucIFrob}{\mathsf{ucIFrob}}
\newcommand{\ucom}{\mathsf{ucom}}
\newcommand{\com}{\mathsf{com}}
\newcommand{\rcom}{\overleftarrow{\com}}
\newcommand{\rlie}{\overleftarrow{\lie}}
\newcommand{\lie}{\mathsf{lie}}
\newcommand{\uCom}{\mathsf{uCom}}

\newcommand{\LLL}{LLS}
\newcommand{\RRR}{RLS}

\newcommand{\Conv}{\mathrm{Conv}}

\newcommand{\CAlg}{\mathcal{CA}\mathit{lg}}
\newcommand{\HLVCat}{\uIFrob_{n,\infty}\text{-}\mathcal{A}\mathit{lg}_\bullet}
\newcommand{\cHLVCat}{\ucIFrob_{n,\infty}\text{-}\mathcal{A}\mathit{lg}_\bullet}
\newcommand{\GXCat}{\Bar^c\GX_n\text{-}\mathcal{A}\mathit{lg}_\bullet}
\newcommand{\dgVect}{\mathit{dg}\mathcal{V}\mathit{ect}}
\newcommand{\sSet}{\mathit{s}\mathcal{S}\mathit{et}}
\newcommand{\dgca}{\mathcal{D}\mathit{gca}}
\newcommand{\coInd}{\mathrm{coInd}}
\newcommand{\Ind}{\mathrm{Ind}}

\renewcommand{\Bar}{\mathtt{B}}
\newcommand{\POp}{\mathsf{P}}
\newcommand{\GCA}{\mathsf{GCA}}

\newcommand{\pBimod}{\mathbf{pBimod}}
\newcommand{\bimod}{\mathbf{bimod}}
\newcommand{\from}{\leftarrow}

\newcommand{\cQ}{\mathcal{Q}}
\newcommand{\cM}{\mathcal{M}}

\newcommand{\cO}{\mathcal{O}}

\newcommand{\Ppd}{\mathbf{Ppd}}
\newcommand{\dash}{\text{-}}

\begin{document}
\title{Homotopy Frobenius structures on the cohomology of a manifold}

\author{Florian Naef}
\address{School of Mathematics\\
Trinity College\\
Dublin 2, Ireland, D02 PN40}
\email{naeff@tcd.ie}

\author{Thomas Willwacher}
\address{Department of Mathematics\\ ETH Zurich\\ R\"amistrasse 101 \\ 8092 Zurich, Switzerland}
\email{thomas.willwacher@math.ethz.ch}

\thanks{
T.W. has been partially supported by the NCCR SwissMAP funded by the Swiss National Science foundation.
}


\begin{abstract}
We show that the category of lax involutive $n$-Frobenius algebras is Quillen equivalent to the category of right comodules of the $n$-Poisson cooperad.
It follows in particular, that the cohomology of a parallelized $n$-manifold is naturally endowed with a homotopy involutive $n$-Frobenius structure extending the rational homotopy type of $M$, solving a long-standing question.
\end{abstract}

\maketitle


\section{Introduction}
Let $A$ be finite dimensional graded vector space. Then a unital $n$-Frobenius algebra structure on $A$ is a unital graded commutative algebra structure on $A$ together with a distinguished element of degree $n$ (the \emph{diagonal}),
\[
\Delta = \sum a' \otimes a'' \in A\otimes A,
\]
such that $\Delta$ is $(-1)^n$-symmetric under switching the factors, and such that for all $a\in A$ we have 
\[
(a\otimes 1)\Delta = (1\otimes a)\Delta,
\]
using the natural graded commutative product on $A\otimes A$.
We define the \emph{Euler class} of $A$ to be the element 
\begin{equation}\label{equ:Euler}
E=\sum a'a''\in A.
\end{equation}
If the Euler class vanishes, then we call the unital $n$-Frobenius strucuture \emph{involutive}. This is always the case if $n$ is odd due to the opposite symmetries of $\Delta$ and the product.
If the diagonal element $\Delta$ happens to be non-degenerate, i.e., of full rank, then the unital $n$-Frobenius structure is also counital, with the counit being the element dual to $1$ under the pairing induced by $\Delta$. In general, however, our Frobenius algebras might have degenerate diagonal elements, even $\Delta=0$.

Now let $M$ be an oriented, but not necessarily closed manifold of dimension $n\geq 2$ with finite dimensional rational cohomology.
Then the cohomology $H^\bullet(M)$ is naturally a unital $n$-Frobenius algebra in the sense above. Here the diagonal $\Delta$ can be taken to be the image of the canonical element of $H^\bullet(M)\otimes H_c^\bullet(M)$ encoding the Poincar\'e duality pairing under the map 
\[
H^\bullet(M)\otimes H_c^\bullet(M)\to H^\bullet(M)\otimes H^\bullet(M).
\]
The Euler class \eqref{equ:Euler} of this unital $n$-Frobenius algebra is the Euler class of $M$, and hence $H^\bullet(M)$ is involutive if the Euler class of $M$ vanishes. 
If $M$ is furthermore closed, then the diagonal element is non-degenerate and the $n$-Frobenius structure is also counital.

There have been attempts and the hope to meaningfully upgrade this (involutive) unital $n$-Frobenius structure on $H^\bullet(M)$ to the cochain level in a natural way, say to some homotopy (involutive) $n$-Frobenius structure on a version of the differential forms $\Omega(M)$, or dually to chains. Homotopy-transferring this structure back to cohomology, we would then get a homotopy Frobenius structure on $H^\bullet(M)$. In contrast to the trivial structure, this structure would in particular encode the rational homotopy type of $M$, in the form of a homotopy commutative algebra structure on $H^\bullet(M)$, that is a subset of the $n$-Frobenius structure.
The homotopy type of the homotopy $n$-Frobenius structure on $H^\bullet(M)$ would then be an invariant of the manifold $M$, upgrading the rational homotopy type. A first negative result as to the existence of a natural homotopy Frobenius structure on $H^\bullet(M)$ has been obtained in \cite{JohnsonFreyd}, albeit in the dimension $n=1$ not considered in this paper. 
If $M$ is closed and simply connected, then a positive answer has been obtained by Lambrechts and Stanley \cite{LambrechtsStanley}, but with no naturality properties. In this paper, we will provide a positive and natural solution in the case that $M$ is parallelized and $n\geq 2$. 

On the other hand, assume now that $M$ is parallelized, and sill not necessarily closed. Then the totality of the configuration spaces of points on $M$ can be made an operadic right module over a version of the framed little disks operad. For example, let $\FM_M(r)$ be the Fulton-MacPherson-Axelrod-Singer version of the configuration space of $r$ distinguishable points on $M$. Then the collection $\FM_M=\{\FM(r)\}_{r}$ is naturally an operadic right module over the Fulton-MacPherson operad $\FM_n$.
Taking (a version of) differential forms of these models we obtain a right comodule $\Omega_\sharp(\FM_M)$ over the Hopf cooperad $\Omega_\sharp(\FM_n)$. The homotopy type of $\Omega_\sharp(\FM_M)$ is then a (framed) diffeomorphism invariant of the manifold $M$, that contains the rational homotopy type of $M$.

The main result of this paper is that the two types of algebraic data, namely a homotopy involutive Frobenius algebra structure on $H^\bullet(M)$ and the rational homotopy type of the configuration spaces of points on $M$ as right $\FM_n$-module are in fact equivalent. This in particular implies the existence of a meaningful homotopy unital $n$-Frobenius structure on $H^\bullet(M)$, for parallelizable manifolds $M$, settling a long standing question.
Note that in the parallelized case the Euler class is automatically zero, so we are looking for involutive $n$-Frobenius structures on $H^\bullet(M)$. 

\subsection{Results}
Let us state our results more precisely. To this end, we need to introduce in particular our encoding of a homotopy (or weak) $n$-Frobenius algebra. Let $\uIFrob_n$ be the properad governing unital, but possibly not counital involutive $n$-Frobenius algebras, i.e., 
\begin{equation}\label{equ:uIFrob}
\uIFrob_n(r, s) = 
\begin{cases}
    \Q[n(r-1)] &\text{if $r\geq 1$, $s\geq 0$} \\
    0 & \text{otherwise},
\end{cases}
\end{equation}
with the action of the symmetric group $S_s$ on the input side being trivial, and on the action of $S_r$ on the output side being with sign if $n$ is odd and trivial otherwise. A non-counital involutive $n$-Frobenius algebra is then a dg vector space $V$ together with morphisms of dg $S_r$-modules 
\[
\uIFrob_n(r, s) \otimes_{S_s} V^{\otimes s} \to V^{\otimes r}
\]
satisfying natural coherence relations.
To relax this notion, consider the PROP envelop $\uifrob_n$ of the properad $\uIFrob_n$.
The PROP $\uifrob_n$ is a category enriched in graded vector spaces, with object set $\mathbb N_0=\{0,1,2,\dots\}$ and graded vector spaces of morphisms (from $s$ to $r$) $\uifrob(r,s)$, which are obtained (essentially) by formal horizontal concatenation of operations of $\uIFrob_n$.
Let 
$$
\uifrob_n\Mod=\Fun(\uifrob_n, \dgVect)
$$ 
be the category of left $\uifrob_n$-modules, i.e., the category of covariant $\dgVect$-enriched functors $\uifrob_n\to \dgVect$. For example, any $\uIFrob_n$-algebra $V$ gives rise to such a functor sending the object $s$ to $V^{\otimes s}$. Moreover, $\uifrob_n$ and $\dgVect$ are symmetric monoidal categories, and $\uIFrob_n$-algebras are equivalent data to strongly monoidal enriched functors $\uifrob_n\to \dgVect$.
We relax the notion of $\uIFrob_n$-algebra by relaxing the strong monoidality assumption. 
First define 
\[
\CAlg \left( \uifrob_n\Mod\right)
= \Fun_{\text{lax-}\otimes}(\uifrob, \dgVect)
\]
to be the category of symmetric lax monoidal functors. These are the same as commutative monoid objects in the category $\uifrob_n\Mod$, equipped with the Day convolution monoidal product. Here we denote the commutative monoid objects in a category $\cD$ by $\CAlg(\cD)$. Now any $F\in \CAlg \left(\uifrob_n\Mod\right)$ comes with natural comparison morphisms 
\begin{equation}
\label{equ:segal}
F(1)^{\otimes r} \to F(r)
\end{equation}
from the monoidal structure. These are isomorphisms if $F$ is a strict $\uIFrob_n$-algebra. We say that $F$ satisfies the \emph{Segal condition} if the morphisms \eqref{equ:segal} are weak equivalences (quasi-isomorphisms) for all $r$, and we call such $F$ \emph{homotopy $\uIFrob_n$-algebras}.

Next, let us turn to the other (configuration space) side. We desire to understand right Hopf comodules over a cochain model of the little $n$-disks operad. By formality of the little disks operad, we may simply take the (unital) $n$-Poisson cooperad $\Pois_n^c=H^\bullet(\FM_n)$ for this model. The attribute ``Hopf'' refers to the fact that we consider $\Pois_n^c$ as a cooperad in the category of dg commutative algebras $\dgca$, snd similarly for the comodules. 
Let $\pois_n^c$ be the co-PROP envelope of $\Pois_n^c$, and similarly let $\pois_n$ be the PROP envelope of the operad $\Pois_n$. Then a right (Hopf) $\pois_n^c$-comodule is in fact the same data as a left (Hopf) $\pois_n$-module since $\pois_n$ is arity-wise finite dimensional. Hence we can equivalently study the latter. 
More precisely, let 
\[
\pois_n\Mod = \Fun(\pois_n, \dgVect)
\]
be the category of plain (non-Hopf) left $\pois_n$-modules, i.e., $\dgVect$-enriched functors $\pois_n\to \dgVect$.
Then the category $\pois_n\Mod$ is equipped with a natural symmetric monoidal structure such that for $F,G\in \pois_n\Mod$
\[
(F\otimes G)(r) = F(r)\otimes G(r).
\]
Now the right Hopf $\Pois_n^c$-comodules we want to study are the same as commutative monoid objects in this monoidal category, i.e., objects in $\CAlg\left(\pois_n\Mod\right)$.
There is furthermore a Segal-type condition on those objects. Namely, for $X$ a right Hopf $\Pois_n^c$-comodule there are natural maps 
\begin{equation}\label{equ:config space type}
X(1)^{\otimes r} \to (\coInd^{\Com^c}_{\Pois_n^c} X) (r)
\end{equation}
where $\coInd^{\Com^c}_{\Pois_n^c} X$ is the homotopy coinduced $\Com^c$-comodule. Then following \cite{WillwacherObstruction} we say that $X$ is of \emph{configuration space type} if the morphisms \eqref{equ:config space type} are weak equivalences (quasi-isomorphisms) for all $r$. As the name suggests, $\Omega_\sharp(\FM_M)$ is of configuration space type, at least in the case where $H^\bullet(M)$ is finite-dimensional.

Now our main result is essentially a derived Morita equivalence for the categories $\uifrob_n\Mod$ and $\pois_n\Mod$, that aligns well with all structures discussed above.
\begin{thm}\label{thm:main 1}
There are a $\pois_n$-$\uifrob_n$-bimodule $\LLL$ and a $\uifrob_n$-$\pois_n$-bimodule $\RRR$ such that the associated functors 
\[
\begin{tikzcd}[column sep=2cm]
     \pois_n\Mod
     \ar[shift left]{r}{\LLL\circ_{\pois_n}(-)} 
     &
     \uifrob_n\Mod 
     \ar[shift left]{l}{\RRR\circ_{\uifrob_n}(-)} 
\end{tikzcd}
\]
form a Quillen equivalence. 
Moreover, the left-adjoint is strong symmetric monoidal and we obtain
an induced Quillen equivalence on the categories of commutative algebra objects 
\[
\CAlg(\pois_n\Mod) \leftrightarrows \CAlg(\uifrob_n\Mod).
\]
Furthermore, under this equivalence the homotopy $\uIFrob_n$-algebras are sent to the right Hopf $\Pois_n^c$-comodules of configuration space type and vice versa.
\end{thm}
The last statement says that the Segal type conditions we can impose on objects of either side match. We note that the subcategories of Segal type objects on both sides represent $\infty$-categories (for instance given by simplicial categories obtained by taking full subcategories of their simplicial localization) that can not be represented by model categories. We thus obtain
\begin{cor}\label{cor:Segal equiv}
    There is an equivalence of $\infty$-categories
    \begin{equation}\label{equ:Segal equiv}
    \CAlg(\uifrob_n\Mod)^\text{Segal} \simeq \CAlg(\pois_n\Mod)^\text{config},
    \end{equation}
    where the categories on either side are the full subcategories of objects satisfying the Segal condition, resp, the configuration space type condition.
\end{cor}

For a parallelized $n$-manifold $M$ of finite type (with $n\geq 2$) the rational model $\Omega_\sharp(\FM_M)$ of the configuration space of points on $M$ combined with the formality of the little disks operad yields a right Hopf $\Pois_n^c$-comodule of configuration space type. Then Theorem \ref{thm:main 1} implies:
\begin{cor}\label{cor:uIFrob alg on H}
Let $M$ be a parallelized finite-type $n$-manifold with $n\geq 2$ and rational cohomology $H^\bullet(M)$. Then there is a natural homotopy $\uIFrob_n$-algebra $H_M$ such that $H_M(r)\simeq (H^\bullet(M))^{\otimes r}$ whose homotopy type encodes the rational homotopy type of $\FM_M$ via \eqref{equ:Segal equiv}.
\end{cor}

The structure is natural with respect to framed embeddings of manifolds. Also, note that by the equivalence \eqref{equ:Segal equiv}, the rational homotopy type of $\FM_M$ can be reconstructed from the homotopy $\uIFrob_n$-algebra $H_M$.

Let us remark on two connections to existing constructions in the literature.
The first is that our definition of homotopy Frobenius algebras differs from the one more prevalent in the literature, namely as an algebra over a cofibrant resolution $\uIFrob_{n,\infty}$ of the properad $\uIFrob_{n}$.
Furthermore, there is a natural simplicial category constructed by Hoffbeck-Leray-Vallette \cite{HoffbeckLerayVallette} which we denote by $\HLVCat$, and which we may take as defining the $\infty$-category of $\uIFrob_{n,\infty}$-modules. The underlying ordinary category of $\HLVCat$ is the category of $\uIFrob_{n,\infty}$-algebras with $\infty$-morphisms between objects. We show the following:

\begin{prop}\label{prop:HLV equiv intro}
    There is an equivalence of $\infty$-categories
    \[
    \HLVCat \simeq \CAlg(\uifrob_n\Mod)^\text{Segal},
    \]
\end{prop}
In particular, note that the homotopy $\uIFrob_n$-algebra structure of Corollary \ref{cor:uIFrob alg on H} can equivalently be encoded as a $\uIFrob_{n,\infty}$-algebra structure on $H^\bullet(M)$.

The second point is to clarify how our present results relate to the work of the second author
\cite{WillwacherObstruction}, who introduced graphical models for general Hopf $\Pois_n^c$-comodules of configuration space type.
In particular, the homotopy type of a Hopf $\Pois_n^c$-comodule of configuration space type is nicely encoded combinatorially by a Maurer-Cartan element in a certain graph complex.
Furthermore, morphisms between such comodules can also be encoded through graph complexes, and there is a simplicial category $\GCA$ packaging both objects and morphisms introduced by Abramyan \cite{Abramyan}.
Notably, the graph complexes involved in these constructions are built from undirected graphs.
On the other hand, $\uIFrob_{n,\infty}$-algebra structures are also naturally encoded as Maurer-Cartan elements in a graph complex, as are the mapping spaces in the Hoffbeck-Leray-Vallette simplicial category $\HLVCat$. However, the latter graph complexes are built from different, in particular directed acyclic graphs.
We show that both theories can be neatly related as follows. 

First, we construct a properad $\Bar^c\GX_n$ built from undirected 
graphs, that comes with a quasi-isomorphism of properads $\Bar^c\GX_n \to \uIFrob_{n,\infty}$, see Corollary \ref{cor:Theta qiso properad} below. 
We denote the Hoffbeck-Leray-Vallette simplicial category of $\Bar^c\GX_n$-algebras by $\GXCat$. This then comes with a natural equivalence of $\infty$-categories $\GXCat \simeq \HLVCat$. 

Second, both of these simplicial categories have subcategories of strongly unital objects $\HLVCat^{su}$ and $\GXCat^\text{su}$, see Section \ref{sec:strongly unital} for the definition. 
In particular, the category $\GXCat^\text{su}$ is almost identical to Abramyan's category $\GCA$, except that we do not impose finiteness conditions on the objects. In summary we have the following result.



\begin{thm}\label{thm:main 2}
    There are simplicial functors 
    \[
    \GCA \to \GXCat^{su}
    \xrightarrow{\sim} \GXCat 
    \xrightarrow{\sim} \HLVCat
    \]
    such that the two right-hand functors are equivalences of $\infty$-categories, and the left-hand functor is fully faithful.
\end{thm}

\begin{rem}
Let us remark on the role of parallelizability of the manifold $M$.
Note that our construction of a natural homotopy $n$-Frobenius structure on the cohomology of $M$ through the configuration spaces $\FM_M$ requires that $\FM_M$ has a natural right $\FM_n$-module structure. This in turn we can a priori only guarantee if $M$ is parallelized, and the right $\FM_n$-module structure depends on the parallelization.
However, note that in fact we only need weaker data than a parallelization. The configuration spaces $\FM_M$ are always a right module over the fiberwise Fulton-MacPherson operad $\FM_n^M$. Thus we only need a trivialization of the latter, and in fact only a rational such trivialization 
\[
\tau: \FM_n^\Q\times M^\Q  \xrightarrow{\sim}  (\FM_n^M)^\Q.
\]
Our construction of a homotopy unital involutive $n$-Frobenius structure is then natural in the pair $(M, \tau)$ consisting of the manifold $M$ and a trivialization as above.
The obstruction for the existence of such a parallelization is the triviality of the classifying map 
\[
M^\Q \to B\mathrm{SO}(n)^\Q \xrightarrow{f} B\mathrm{Aut}^h(\FM_n^\Q).
\]
The arrow $f$ has been studied over the reals in \cite{KhoroshkinWillwacher} and found to be nontrivial only on the Euler class (for even $n$), respectively the top Pontryagin class (for odd $n$). The top Pontryagin class of $M$ always vanishes (in up to $n$-dimensional families) by degree reasons, and thus the only obstruction to the existence of a trivialization $\tau$ is in fact the vanishing of the Euler class. We also note that explicit models for $(\FM_n^M)^\R$ have been constructed in \cite{CDIW}, which indeed trivialize if the Euler class vanishes. We expect that the same results also hold rationally.
\end{rem}

The above result constructs a homotopy $\uIFrob_n$-algebra structure on $H^\bullet(M)$ that is in particular natural in framed codimension zero embeddings. Moreover, if $M$ is closed, then the corresponding diagonal class $\Delta \in H^\bullet(M) \otimes H^\bullet(M)$ defines a non-degenerate pairing and we obtain a unique counit (or fundamental class) defined to be dual to the unit. Moreover, any map that preserves the counit is automatically an equivalence. We show the following homotopical version of that statement. Let us call a homotopy $\uIFrob_n$-algebra $A$ \emph{non-degenerate} if the diagonal class $\Delta \in A \otimes A$ defines a non-degenerate inner pairing on $H^\bullet(A)$. Also, let $\ucIFrob$ denote the properad governing unital, counital involutive Frobenius algebras, i.e. $\ucIFrob_n(r,s) = \Q[n(r-1)]$ for all $r$ and $s$. We then show the following: 
\begin{thm}[{Corollary \ref{cor:uc equiv} below}]
The forgetful functor $\cHLVCat \to \HLVCat$ identifies $\cHLVCat$ with the groupoid core of the full subcategory of non-degenerate unital involutive Frobenius algebras, that is, we obtain an equivalence of $\infty$-categories ($\infty$-groupoids)
\[
\cHLVCat \to \HLVCat^{\text{non-deg}, \simeq}.
\]
\end{thm}

\subsection*{Structure of the paper}
We establish basic notation and the homotopy theoretic and model categorical background in Section \ref{sec:background}.
The proof of our main Theorem \ref{thm:main 1} is given in Section \ref{sec:proof main 1}.
Section \ref{sec:htpy transfer} contains a rigidification result for homotopy $\uIFrob_n$-algebras. This is in particular used in the proof of Proposition \ref{prop:HLV equiv intro}.
In Section \ref{sec:GX} we construct our "undirected graph" model $\Bar^c\GX_n$ for $\uIFrob_{n,\infty}$.
In Section \ref{sec:HLV} the construction of the (equivalent) Hoffbeck-Leray-Vallette simplicial categories $\HLVCat$ and $\GXCat$ is recalled.
Proposition \ref{prop:HLV equiv intro} comparing the Hoffbeck-Leray-Vallette  category to ours is then shown in Section \ref{sec:HLV comparison}.
The notion of strong unitality is introduced in Section \ref{sec:strongly unital}, and Theorem \ref{thm:main 2} is shown there.
The Frobenius algebras considered until this point are unital but not counital.
The final Section \ref{sec:counital} contains a discussion of the Frobenius algebras that are both unital and counital and their mapping spaces.

\subsection*{Acknowledgements}
We drew inspiration from works of Marko \v Zivkovi\'c, whom we thank for helpful discussions during the early stages of this project.
We also thank Ben Knudsen, Manuel Krannich, Alexander Kupers and Jan Steinebrunner for helpful discussions.

\section{Background, notation and model categories}
\label{sec:background}
\subsection{Categorial notation}
Let $\cC$ be a category and let $X$ and $Y$ be objects of $\cC$. Then we write $\Mor_{\cC}(X,Y)$ for the set of morphisms in $\cC$, or also just $\Mor(X,Y)$ if the category is clear from the context.
If the category $\cC$ is a dg category, as is often the case in this paper, we write $\Hom_{\cC}(A,B)$ or just $\Hom(A,B)$ for the dg vector space of morphisms. Hence $\Mor(A,B)$ are the degree zero cocycles in $\Hom(A,B)$.
If the category $\cC$ is enriched over simplicial sets we write $\Map_{\cC}(A,B)$ or $\Map(A,B)$ to denote the simplicial mapping spaces.

\subsection{Several (pr)operads and PROPS}
Let $\Com$ be the operad governing non-unital commutative algebras and let $\uCom$ be the operad governing unital commutative algebras. We denote by $\Com_n:=\Com\{n\}$ its degree-shifted version defined such that 
\[
\Com_n(r) = 
\begin{cases}
    \Q[(n-1)(r-1)]\otimes \sgn_r^{\otimes n} & \text{for $r\geq 1$} \\
    0 & \text{for $r= 0$}
\end{cases},
\]
where $\sgn_r$ is the sign representation of $S_r$ and $\sgn_r^{\otimes n}$ is the sign representation for odd $n$ and the trivial representation for even $n$.
We denote by $\Lie$ the Lie operad and by $\Lie_n:=\Lie\{n-1\}$ its degree-shifted version.
The notation $\Pois_n$ shall refer to the operad governing unital $n$-Poisson algebras. $\Pois_n$ can be identified with the homology operad of the little $n$-disks operad. We furthermore have the isomorphism of symmetric sequences $\Pois_n \cong \uCom \circ \Lie_n$, where $\circ$ is the plethysm product on symmetric sequences.
The dual cooperads are denoted by $\Com^c$, $\Lie^c$ etc.



For $\mathsf{P}$ an operad or properad, we denote by lowercase letters (e.g., $\mathsf{p}$) its PROP envelope.
It is defined such that
\[
\mathsf{p}(r,s)
=
\bigoplus_k\left( 
\bigotimes_{r_1+\cdots+r_k=r\atop s_1+\cdots+s_k=s } \Ind_{S_{r_1}\times \cdots \times S_{r_k}\times S_{s_1}\times\cdots\times S_{s_k}}^{S_r\times S_s}
\mathsf P(r_1,s_1) \otimes \cdots \otimes\mathsf P(r_k,s_k)
\right)_{S_k}.
\]
Note that we follow standard properadic conventions in writing the number of inputs ($s$) as the second argument above, whereas $r$ is the number of outputs.
Note that given $\mathsf{P}$ and $\mathsf{Q}$ with envelopes $\mathsf{p}$ and $\mathsf{q}$, respectively, the envelope of $\mathsf{P} \circ \mathsf{Q}$ is given by the bisymmetric sequence
\[
(r,s) \mapsto \bigoplus_{k} \mathsf p(r,k) \otimes_{S_k} \mathsf q(k,s).
\]
We will abusively denote this bisymmetric sequence by $\mathsf{p} \circ \mathsf{q}$ as well. In particular, envelopes of properads become algebras in bisymmetric sequences under this (disconnected) composition product. Since this is now linear in both arguments, we can consider left modules, right modules and bimodules. With respect to this composition product bisymmetric sequences have inner Homs, namely, given bisymmetric sequences $x$, $y$ and $z$ there exists
\[
\Hom_{A-B}(x \circ_C y, z) = \Hom_{A-C}(x, \HHom_{-B}(y,z)) = \Hom_{C-}(y, \HHom_{A-}(x,z)),
\]
where for instance
\[
\HHom_{-s}(y,z)(r,s) = \prod_{k} \Hom_{S_k}(y(s,k), z(r,k))
\]
for $s$ the envelope of the trivial operad and
\[
\HHom_{-A}(y,z) = \mathrm{eq}(\HHom_{-s}(y,z) \rightrightarrows \HHom_{-s}(y \circ A, z) ).
\]

We consider $\mathsf{p}$ as a $\dgVect$-enriched symmetric monoidal category with object set $\mathbb N_0$.

For $\mathsf{P}$ a properad or PROP we denote by $\overleftarrow{\mathsf{P}}$ its opposite, defined such that $\overleftarrow{\mathsf{P}}(r,s) = \mathsf P(s,r)$.

\subsection{Monoidal structures and Day convolution}
For $\cC$ a small (symmetric) monoidal $\dgVect$-enriched category the category of enriched functors 
\[
\Fun(\cC, \dgVect)
\]
carries a (symmetric) monoidal structure by Day convolution. Concretely, the monoidal product of functors $F,G:\cC\to \dgVect$ is defined by left Kan extension 
\[
\begin{tikzcd}
\cC \times \cC \ar{r} \ar{dr}[swap]{\otimes}& \dgVect\times \dgVect \ar{r} & \dgVect \\
& \cC \ar[dashed]{ur}[swap]{F\otimes G}& 
\end{tikzcd}.
\]
The (commutative) monoid objects in $\Fun(\cC, \dgVect)$ can be identified with lax monoidal enriched functors $\cC\to \dgVect$.
As for the left Kan extension, there are fairly explicit formulas for the Day convolution. 

Next suppose that $\cC$ is not just $\dgVect$-enriched, but enriched over dg (cocommutative) coalgebras. 
This is true in particular if $\cC$ is the PROP envelope of a Hopf operad like $\ucom$, $\com$ or $\pois_n$.
Then there is a second monoidal structure on $\Fun(\cC, \dgVect)$ that we call the object-wise monoidal structure, defined such that for $F,G\in \Fun(\cC, \dgVect)$
\[
(F\otimes G)(r) \cong F(r) \otimes G(r).
\]

Below we shall need the following statement.
\begin{lemma}\label{lem:com day pointwise}
    For any $F,G\in \ucom\Mod$ the monoidal product via Day convolution $F\otimes_{\text{Day}} G$ agrees with the object-wise monoidal product  $F\otimes_{\text{obj.-wise}} G$, i.e.,
    \[
    (F\otimes_{Day} G)(r) \cong F(r) \otimes G(r).
    \]
\end{lemma}
\begin{proof}
To check the statement we first unpack the definition of Day convolution as a left Kan extension. 
For $A,B\in\ucom\Mod$ we obtain a functor $A\boxtimes B: \ucom\times \ucom \to \dgVect$ as the composition 
\[
\ucom\times \ucom \xrightarrow{A\times B}  \dgVect\times \dgVect 
\xrightarrow{\otimes} \dgVect.
\]
The Day convolution $A\otimes_{Day}B$ is the left Kan extension in the diagram
\[
\begin{tikzcd}
    \ucom \ar{r}{A\otimes_{Day}B} & \dgVect \\
    \ucom\times \ucom \ar{ur}{A\boxtimes B} \ar{u} & 
\end{tikzcd}.
\]
This in turn means that we have an adjunction 
\[
\Hom_{\ucom\Mod}(A\otimes_{Day}B,C) \cong 
\Hom_{\ucom\times\ucom\Mod}(A\boxtimes B, R(C) ),
\]
where the right-adjoint $R(-)$ is restriction along the monoidal product $\ucom\times\ucom\to \ucom$. Explicitly, for $C\in \ucom\Mod$ we have on objects $R(C)(p,q)=C(p+q)$ and on morphisms $x\in \ucom(r,s)$, $y\in \ucom(r',s')$
\[
R(C)(x\times y) = C(x\boxtimes y),
\]
using the "horizontal" concatenation $\boxtimes$ we have on the PROP $\ucom$.
Now, having made explicit the definition of Day convolution we have to check that $A\otimes_{Day}B=A\otimes B:=A\otimes_{obj.-wise}B$ agrees with the object-wise monoidal product. We do this by checking explicitly the adjunction relation 
\[
\Hom_{\ucom\Mod}(A\otimes B,C) \stackrel{?}{\cong} 
\Hom_{\ucom\times\ucom\Mod}(A\boxtimes B, R(C) )
\]
by constructing a natural bijection between both sides.
Given any $\phi: A\boxtimes B \to R(C)$ we define the corresponding morphism $\phi^\sharp: A\otimes B\to C$ such that for $a\in A(r), b\in B(r)$
\[
\phi^\sharp(a\otimes b) := \theta_r (\phi(a\times b)),
\]
where $\theta_r\in \ucom(r,2r)$ is the morphism taking the product of input $i$ and $i+r$ for $i=1,\dots,r$.

Conversely, for any $\psi: A\otimes B\to C$ we define the corresponding morphism $\psi^\sharp: A\boxtimes B \to R(C)$ such that for $a\in A(r), b\in B(s)$
\[
\psi^\sharp(a\times b) = \psi(a 1^s\otimes 1^r b),
\]
where $a 1^s$ (respectively $1^r b$) is obtained from $a$ (resp. $b$) by applying the operation in $\ucom(r+s,r)$ (resp. $\ucom(r+s,s)$) that creates $s$ units after (resp. $r$ units before) the input.

We then have to verify four statements:

1) Functoriality of $\phi^\sharp$: For any $x\in \ucom(r,s)$ denote the coproduct in sum-free Sweedler notation by $x'\otimes x''$.
Then note that we have
\[
x \circ \theta_s = \theta_r \circ (x'\boxtimes x'').
\]
Now we can verify:
\[
x\phi^\sharp(a\otimes b) = x \theta_s(\phi(a\times b))
=
\theta_r (x'\boxtimes x'')\phi(a\times b)
=
\theta_r\phi(x'(a)\times x''(b)b)
\phi^\sharp(x'(a)\otimes x''(b)b)
=
\phi^\sharp(x(a\otimes b)).
\]

2) Functoriality of $\psi^\sharp$:
For $x\in \ucom(r,s)$ and $y\in \ucom(p,q)$ we compute 
\[
(x\boxtimes y)\psi^\sharp(a\times b)
=
(x\boxtimes y)\psi(a1^q\otimes 1^sb)
=
\psi((x\boxtimes y)(a1^q\otimes 1^sb))
=
\psi((x(a)1^q\otimes 1^sy(b)))
=
\psi^\sharp(x(a)\times y(b)).
\]

3) $(\phi^\sharp)^\sharp=\phi$: We compute 
\[
(\phi^\sharp)^\sharp(a\times b)
=
\phi^\sharp(a1^s\otimes 1^rb)
=
\theta_{r+s} \phi(a1^s\times 1^rb)
=
\theta_{r+s} u_{s,r}\phi(a\times b)
=
\phi(a\times b),
\]
where the morphism $u_{s,r}\in \ucom(2s+2r,r+s)$ pads with $r+s$ units and satisfies $\theta_{r+s} u_{s,r}=\mathit{id}$.

4) $(\psi^\sharp)^\sharp=\psi$:
We compute 
\[
(\psi^\sharp)^\sharp(a\otimes b)
=
\theta_r(\psi^\sharp(a\times b))
=
\theta_r(\psi(a1^r \otimes 1^rb))
=
\psi( \theta_r(a1^r \otimes 1^rb))
=
\psi(a\otimes b),
\]
using that $\theta_r(a1^r \otimes 1^rb)=a\otimes b$.
\end{proof}

\subsection{Model category structures}
Recall that the PROPS $\ucom$, $\pois_n$ and $\uifrob_n$ are (small) dg-categories with object set $\mathbb{N}_0$. By a result of Keller \cite[Theorem 3.2]{Keller} we can endow $\pois_n\Mod$ and $\uifrob_n\Mod$ with either the projective or injective model structure. We will endow $\ucom\Mod$ and $\pois_n\Mod$ with the injective model structures.
Both categories are symmetric monoidal. We use that $\pois_n$ and $\ucom$ are enriched in coalgebras, so we endow it with the object-wise monoidal structure. For $\ucom\Mod$ this object-wise monoidal structure is the same as the the Day-convolution product by Lemma \ref{lem:com day pointwise}.

\begin{prop}
    $\ucom\Mod$ and $\pois_n\Mod$ are symmetric monoidal model categories.
\end{prop}
\begin{proof}
We have to verify the pushout-product axiom and the unit axiom. Here the latter is automatic since all objects are cofibrant.
For the former, let $f,g$ be two cofibrations in $\ucom\Mod$ (respectively $\pois_n\Mod$). These are objectwise injective morphisms.
Then since the tensor product is computed objectwise and $\dgVect$ is a monoidal model category, we have that the pushout-product $f\square g$ is again a cofibration, and furthermore acyclic if $f$ or $g$ is.
\end{proof}

Let $\iota: \ucom\to \uifrob_n$ be the natural inclusion of PROPs. We have a free/forgetful adjunction ("base change")
\begin{equation}\label{equ:iotaiota}
\iota_!:=\rcom_n\circ- \colon \ucom\Mod \leftrightarrows \uifrob_n\Mod \colon \iota^*.
\end{equation}

We define a cofibrantly generated model category structure on $\uifrob_n\Mod$ by right transfer along the above adjunction, where $\ucom\Mod$ is equipped with the injective model structure. That is, the weak equivalences are the quasi-isomorphisms, the fibrations are those morphisms that are fibrations of $\ucom$-modules, and the generating (acyclic) cofibrations are obtained by applying $\ucom\circ-$ to the generating (acyclic) cofibrations of $\ucom\Mod$.
\begin{lemma}
\label{lem:uifrob model}
    The model category structure on $\uifrob_n\Mod$ obtained from the injective model category structure on $\ucom\Mod$ by right transfer along the forgetful functor
    $\iota^*:\uifrob_n\Mod\to \ucom\Mod$ is well-defined.
\end{lemma}
\begin{proof}
    We first note that all colimits in $\ucom\Mod$ and $\pois_n\Mod$ are created object-wise in $\dgVect$. 
    In particular, the forgetful functor $\iota^*$ preserves all colimits. 
    Furthermore, the acyclic cofibrations in $\ucom\Mod$ are the injective quasi-isomorphisms. But applying the left-adjoint $\iota_!=\rcom_n\circ-$ sends injective quasi-isomorphisms to injective quasi-isomorphisms.
    Hence by \cite[Theorem 7.44 and Remark 7.45]{HeutsMoerdijk} the transferred model structure is well-defined.
\end{proof}

The category $\uifrob_n\Mod$ is equipped with the symmetric monoidal product by Day convolution, using that $\uifrob_n$ is a symmetric monoidal category. Since the inclusion $\iota:\ucom\to \uifrob_n$ is in particular a strong monoidal functor, the functor $\iota_!$ is strong monoidal as well, and consequently its right adjoint $\iota^*$ is lax monoidal (see \cite[Proposition 1]{DayStreet}).

\begin{lemma}\label{lem:uifrob mon model}
    The category $\uifrob_n\Mod$ with the model structure of Lemma \ref{lem:uifrob model} and equipped with the symmetric monoidal product by Day convolution is a symmetric monoidal model category, and \eqref{equ:iotaiota} is a symmetric monoidal Quillen adjunction. 
\end{lemma}
\begin{proof}
We have to verify the pushout-product axiom and the unit axiom for $\uifrob_n\Mod$.
It suffices to check the pushout-product axiom on generating cofibrations $\iota_! i, \iota_! j$, for $i,j$ generating cofibrations for $\ucom\Mod$.
Then we have 
\[
(\iota_!i) \square (\iota_!i) = \iota_! (i\square j)
\]
since $\iota_!$ is strong monoidal and preserves colimits.
Hence the left-hand morphism is a cofibration, and furthermore acyclic if $i$ or $j$ is.

The unit axiom is automatic since the monoidal unit of $\ucom$ is cofibrant (as is any other object) and $\iota_!$ is strong monoidal, so that the monoidal unit of $\uifrob_n\Mod$ is also cofibrant.
\end{proof}

As a consequence, we may also endow the categories $\CAlg(\pois_n\Mod)$ and $\CAlg(\uifrob_n\Mod)$ of commutative algebra objects in $\pois_n\Mod$ and $\uifrob_n\Mod$ with model category structures.

\begin{prop}
$\CAlg(\pois_n\Mod)$ and $\CAlg(\uifrob_n\Mod)$ carry model structures, right-transferred from the underlying monoidal categories.
\end{prop}
\begin{proof}
We check that in both cases the hypothesis of \cite[Theorem 5.11]{PavlovScholbach} are satisfied. That is
\begin{itemize}
    \item $\mathcal{C}$ is combinatorial and weak equivalences are closed under transfinite compositions.
    \item $\mathcal{C}$ symmetric is $h$-monoidal.
\end{itemize}
Since we are in characteristic zero, symmetric $h$-monoidality is implied by $h$-monoidality.
For $\pois_n\Mod$, we obtain $h$-monoidality because everything is computed/defined on the underlying chain complex.
It is combinatorial by \cite[Theorem 2.23]{Bayeh}, since everything in sight is presentable and co-anodyne maps are split-surjections with acyclic kernels, hence equivalences.

For $\uifrob_n\Mod$ we apply 
\cite[Proposition 1.15]{BataninBerger} by noting that the internal hom detects weak equivalences since the forgetful functor to symmetric sequences in $\dgVect$ has a right adjoint. In other words, there are coinduced $R(i) \in \uifrob_n\Mod$ such that $\HHom_{\uifrob_n}(X, R(i))(0) = X(i)$.
\end{proof}

\section{The proof of Theorem \ref{thm:main 1}}
\label{sec:proof main 1}

\subsection{Universal Lambrechts-Stanley construction and right adjoint $R$}
Given an involutive (non-counital) $n$-Frobenius algebra $A$, Lambrechts-Stanley \cite{LambrechtsStanleyModel} construct a $\pois_n^c$-right comodule (equivalently a $\pois_n$-left module) in dgcas given by
\[
LS_A(r) := A^{\otimes r} \otimes \Pois_n^c(k) / \langle \pi_i^*(a)\omega_{ij} = \pm \pi_j^*(a)\omega_{ij}, d\omega_{ij} = \pi_{ij}^*(\Delta(1)) \rangle
\]
We refer to \cite[Section 5.4]{WillwacherObstruction} and \cite[Section 1.8]{Idrissi} for the verification that this is indeed a Hopf $\pois_n$-module. Moreover, $LS_A$ is cofree as a $\lie_n^c$-comodule, that is, the natural map
\[
LS_A \overset{\cong}{\to} \rlie_n^c \circ A^{\otimes \bullet}
\]
is an isomorphism after forgetting the differential, where we denote by $A^{\otimes \bullet}$ the symmetric sequence whose $r$-ary piece is $A^{\otimes r}$. Under this isomorphism, the differential can be described as taking reduced cocomposition in $\lie_n^c$, projecting to the cogenerators (which are the generators of $\com_n$) and then  acting on $A$,
\[
\rlie_n^c \circ A^{\otimes \bullet}
\to
\rlie_n^c \circ \rlie_n^c \circ A^{\otimes \bullet}
\to
\lie_n^c \circ \rcom_n \circ A^{\otimes \bullet}
\to 
\lie_n^c \circ A^{\otimes \bullet}.
\]
In other word, we obtain an isomorphism
\[
LS_A \overset{\cong}{\to} (\rlie_n^c \circ_\tau \rcom_n) \circ_{\rcom_n} A^{\otimes \bullet}
\]
where
\[
(\rlie_n^c \circ_\tau \rcom_n) = (\rlie_n^c \circ \rcom_n , d_\tau)
\]
is the (opposite of the) Koszul complex of $\com_n$.
Note that the above construction goes through for an involutive (non-counital) $n$-Frobenius algebra in any symmetric monoidal dg-category that is tensored over chain complexes, we can apply it to the "universal example", namely to $A = \uifrob_n(1,-)$ in $\mathcal{M}od\text{-}\uifrob_n$ to obtain the following
\begin{prop}
The dg $\lie_n$-$\uifrob_n$ bimodule
\[
\RRR := (\rlie_n^c \circ_\tau \rcom_n) \circ_{\rcom_n} \uifrob_n
\]
extends to an algebra in $\pois_n$-$\uifrob_n$ bimodules. In other words, it induces a lax symmetric monoidal functor
\begin{align*}
    R \colon \uifrob_n\Mod &\to \pois_n\Mod \\
    X &\mapsto \RRR \circ_{\uifrob_n} X.
\end{align*}
\end{prop}
\hfill\qed

For later use let us also note the following:
\begin{lemma}\label{lem:R creates weq}
The functor $R$ creates weak equivalences, i.e., a morphism $f$ between $\uifrob_n$-modules is a quasi-isomorphism if and only if $R(f)$ is a quasi-isomorphism.
\end{lemma}
\begin{proof}
It suffices to check that for any $\uifrob_n$-module $X$ we have $H(X)=0$ iff $H(R(X))=0$. (To see this take for $X$ the cone of $f$.)
Note that as a dg vector space 
\[
R(X) = (\rlie_n^c \circ X, d_\tau+d_X),
\]
where $d_X$ is the internal differential on $X$ and $d_\tau$ is induced from the Koszul differential as described above.
We endow $R(X)$ with an ascending exhaustive filtration by the number of $\lie_n^c$-cogenerators. Since $d_{\tau}$ reduces the number of these cogenerators by one, the associated graded is given by
\[
\gr R(X) = (\rlie_n^c \circ X, d_X).
\]
We hence obtain a spectral sequence 
\[
E^1 = R(H(X)) \Rightarrow H(R(X)).
\]
It is hence clear that if $H(X)=0$ then $H(R(X))=0$. Conversely, suppose $H(R(X))=0$ and assume by contradiction that $H(X)\neq 0$. Let $r$ be smallest such that $H(X)(r)\neq 0$. Then $H(X)(r)$ is a direct summand of $R(H(X))$, closed under the differential, and it cannot be hit by any higher differentials, since they would need to originate from lower arity pieces. Hence $H(R(X))\neq 0$, a contradiction.
\end{proof}

\subsection{The bimodule $\LLL$ and left-adjoint $L$}

From the description of the functor $R$ above we obtain that is has a left-adjoint given by
\begin{align*}
    L \colon \pois_n\Mod &\to \uifrob_n\Mod \\
    Y &\mapsto \HHom_{\uifrob_n}(\RRR, \uifrob_n) \circ_{\pois_n} X = \HHom_{\rcom_n}( (\rlie_n^c \circ_\tau \rcom_n), \uifrob_n) \circ_{\pois_n} X.
\end{align*}
We define the $\uifrob_n$-$\pois_n$ bimodule
\[
\LLL := (\uifrob_n \circ_\ucom \pois_n, d_\sigma) := \HHom_{\uifrob_n}(\RRR, \uifrob_n),
\]
so that 
\[
L(X) = \LLL\circ_{\pois_n} X.
\]
Before proceeding, let us make the combinatorial form of the differential $d_\sigma$ more explicit, using that
\[
\LLL=\uifrob_n \circ_\ucom \pois_n = \overleftarrow{\com_n}
\circ \ucom \circ \lie_n = \overleftarrow{\com_n}
\circ\pois_n.
\]

Elements can be thought of combinatorially as two-level forests
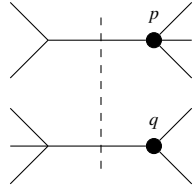
\[
\begin{tikzpicture} 
\draw[dashed] (0,-1)--(0,1);
\coordinate (v1) at (-.7,.7) {}; 
\coordinate (v2) at (-.7,-.7) {}; 
\node[int, label=90:{$\scriptstyle p$}] (w1) at (.7,.7) {}; 
\node[int, label=90:{$\scriptstyle q$}] (w2) at (.7,-.7) {}; 
\draw (v1) edge +(-.5,.5)  edge +(-.5,-.5);
\draw (v2) edge +(-.5,.5) edge +(-.5,0) edge +(-.5,-.5);
\draw (w1) edge +(.5,.5) edge +(.5,0) edge +(.5,-.5);
\draw (w2) edge +(.5,.5)  edge +(.5,-.5);
\draw (w2) edge (v2) (w1) edge (v1);
\end{tikzpicture}
\]
with the right-hand vertices decorated by $\Pois_n$-elements ($p$ and $q$), and the left-hand vertices decorated by $\Com_n$-elements (or, that is, being undecorated). Now the differential $d_\sigma$ sums over all ways of selecting two trees in the forest, fusing the vertices to the left and to the right, while applying the product to the decorations on the left and a bracket to the right:
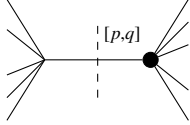
\[
\begin{tikzpicture} 
\draw[dashed] (0,-.5)--(0,.5);
\coordinate (v1) at (-.7,0) {}; 
\node[int, label=93:{$\scriptstyle [p,q]$}] (w1) at (.7,0) {}; 
\draw (v1) edge +(-.5,.8)  edge +(-.5,.4);
\draw (v1) edge +(-.5,-.2) edge +(-.5,-.5) edge +(-.5,-.8);
\draw (w1) edge +(.5,-.8)  edge +(.5,-.4);
\draw (w1) edge +(.5,.2) edge +(.5,.5) edge +(.5,.8);
\draw (w1) edge (v1);
\end{tikzpicture}.
\]

As the left-adjoint of a lax monoidal functor, $L$ is automatically oplax monoidal. We shall show the following stronger statement.

\begin{lemma}\label{lem:L strong monoidal}
The functor $L$ is strong monoidal, that is, the oplax transformation
\[
L(Y \otimes Z) \to L(Y) \otimes L(Z)
\]
is an isomorphism.
\end{lemma}
\begin{proof}
It suffices to check the statement for the underlying graded vector spaces, that is, we may set all differentials to zero.
We note that after forgetting the differential, $LS_A$ only depends on the commutative algebra structure on $A$. We thus obtain that $R$ naturally factors as the composition of the lax symmetric monoidal functors of restriction and coinduction
\[
\uifrob_n\Mod \xrightarrow{\text{Res}} \ucom\Mod \xrightarrow{\text{Coind}} \pois_n\Mod.
\]
In particular, the left adjoint is the composition of corestriction and induction
\[
\pois_n\Mod \xrightarrow{\text{Cores}}  \ucom\Mod \xrightarrow{\text{Ind}}  \uifrob_n\Mod.
\]
The corestriction is the identity on objects and obviously strictly monoidal with respect to the object-wise monoidal structures used for $\pois_n\Mod$ and $\ucom\Mod$.
We furthermore use that by Lemma \ref{lem:com day pointwise} on $\ucom\Mod$ the pointwise monoidal structure is the same as Day convolution, so that the induction $\ucom\Mod \to \uifrob_n\Mod$ is strong monoidal as well, see Lemma \ref{lem:uifrob mon model}. 
\end{proof}


\subsection{Quillen equivalence}

\begin{prop}\label{prop:quillen equivalence}
    The functors $L$ and $R$ yield a Quillen equivalence
    \[
    L \colon \pois_n\Mod \rightleftarrows \uifrob_n\Mod : R,
    \]
    where we consider $\pois_n\Mod$ with the injective model structure and $\uifrob_n\Mod$ with the model structure of Lemma \ref{lem:uifrob model}.
    The left-adjoint $L$ is strong symmetric monoidal, and the right adjoint $R$ is lax symmetric monoidal.
    Furthermore, both functors create weak equivalences.
\end{prop}
\begin{proof}
We already showed the monoidality statement in Lemma \ref{lem:L strong monoidal} and the fact that $R$ creates all weak equivalences in Lemma \ref{lem:R creates weq}.
We next show that $L$ preserves arbitrary weak equivalences (quasi-isomorphisms) by a similar argument:
Let $f:X\to Y$ be a quasi-isomorphism in $\pois_n\Mod$.
We explicitly have that 
\[
L(X) = (\rcom_n \circ X, d_X+ d_\sigma),
\]
with the underlying symmetric sequence
\[
(\rcom_n \circ X)(r) = \bigoplus_{s\leq r} 
\rcom(r,s) \otimes_{S_s} X(s),
\]
and the part of the differential $d_X$ induced by the differential on $X$, while $d_\sigma$ has the form, for $C \otimes x\in \rcom(r,s) \otimes_{S_s} X(s)$  
\[
d_\sigma (C \otimes x) 
=
\sum_{1\leq i<j\leq s}
(C\cdot \bar c_{ij}) \otimes (b_{ij} \cdot x),
\]
where $\bar c_{ij}\in \rcom_n(2,1)$ is the cocommutative coproduct generator and $b_{ij}\in \pois_n(1,2)$ is the Poisson bracket generator.
Let us endow $L(X)$ with the ascending exhaustive filtration such that 
\[
\mF^p L(X) (r) = \bigoplus_{s\leq p } 
\rcom(r,s) \otimes_{S_s} X(s),
\]
and similarly for $y$. 
Then the associated graded $\gr L(f)$ is the morphism (in arity $r$)
\[
(\bigoplus_{s\leq r } 
\rcom(r,s) \otimes_{S_s} X(s), d_X)
\to 
(\bigoplus_{s\leq r } 
\rcom(r,s) \otimes_{S_s} Y(s), d_Y)
\]
induced by $f$. This is obviously a quasi-isomorphism by the Künneth formula, and hence so is $L(f)$.

Next, to see that the adjunction is Quillen it remains to verify that if $f:X\to Y$ is some (possibly non-acyclic) cofibration in $\pois_n\Mod$, i.e., an injective morphism, than $R(f):R(X)\to R(Y)$ is a cofibration in $\uifrob_n\Mod$.
To see this we write $R(f)$ as a transfinite composition
\[
R(X) := A_0 \to A_1\to A_2\to \cdots 
\]
with $A_k = R(X) \sqcup_{\mF^p R(X)} \mF^k R(Y)$.
Let $U=\Q\oplus \Q[-1]$ be the two-dimensional complex with trivial cohomology.
Then each morphism $A_{k}\to A_{k+1}$ fits into a pushout square 
\[
\begin{tikzcd}
  \rcom_n \circ (\mF^{k+1} Y[-1] ) \ar{r}\ar{d}  & A_k \ar{d} \\
  \rcom_n \circ (\mF^{k+1} Y \otimes U)\ar{r}  & A_{k+1}
\end{tikzcd}.
\]
The left-hand vertical arrow is a cofibration, as it is in the image of the Quillen left adjoint $\iota_!$ of \eqref{equ:iotaiota}.
We conclude that all arrows are cofibrations in $\uifrob_n\Mod$, and hence so is the transfinite composition.

Finally, we verify that $L \dashv R$ is a Quillen equivalence. Since $R$ creates weak equivalences by Lemma \ref{lem:R creates weq}, it suffices to verify that the (non-derived) unit 
\[
X \to R(L(X)) 
= \RRR \circ_{\uifrob_n} \LLL \circ_{\pois_n} X
\cong (\rlie_n^c \circ \rcom_n \circ X, d_\tau + d_\sigma + d_X)
\]
is a weak equivalence for all $X\in \pois_n\Mod$ (cf. \cite[Lemma 3.3]{ErdalIlhan}).
Note that the $(\rlie_n^c \circ \rcom_n , d_\tau)$ is the Koszul complex with trivial ("unit") cohomology. Hence by a spectral argument using the same filtration as above it is immediate that $X \to R(L(X))$ is a quasi-isomorphism.

Also note that since the adjunction is a Quillen equivalence and the right-adjoint creates weak equivalences, the left-adjoint also creates weak equivalences automatically.
\end{proof}

From Proposition \ref{prop:quillen equivalence} we immediately obtain:
\begin{cor}
The functors $(L,R)$ induce a Quillen equivalence on the categories of commutative algebra objects 
\[
L \colon \CAlg(\pois_n\Mod) \leftrightarrows \CAlg(\uifrob_n\Mod) \colon R.
\]    
\end{cor}

\subsection{Segal condition}
We next consider the Segal conditions of the introduction and show the remaining statements of Theorem \ref{thm:main 1}, thus finishing its proof.
Note that the Segal conditions for objects in $\CAlg(\uifrob_n\Mod)$ and for objects in $\CAlg(\pois_n\Mod)$ can both be stated in the category $\CAlg(\ucom\Mod)$. More precisely, we have a diagram of categories and functors
\begin{equation}\label{equ:Segaldiagram 1}
\begin{tikzcd}
\pois_n\Mod \ar[shift left, dashed]{rr}{L}
\ar[shift left]{dr}{j_*}
&
&
\uifrob_n\Mod \ar[shift left]{ll}{R}
\ar[shift left]{dl}{\iota^*}
\\
& \ucom\Mod 
\ar[shift left, dashed]{ul}{j^*} 
\ar[shift left, dashed]{ur}{\iota_!} &
\end{tikzcd}.
\end{equation}
Here $(L,R)$ are as in Proposition \ref{prop:quillen equivalence}, $(\iota_!,\iota^*)$ are induction and restriction along the inclusion $\iota:\ucom \to \uifrob_n$, and $(j^*,j_*)$ are (co)restriction and coinduction along the projection $j:\pois_n\to \ucom$. The left-adjoints are indicated by dashed arrows. 
\begin{lemma}
The diagrams of left-adjoints and right-adjoints in \eqref{equ:Segaldiagram 1} commute.
\end{lemma}
\begin{proof}
It suffices to show the statement for the left-adjoints. For $X\in \ucom\Mod$ we have
\[
L(j^* X) = (\rcom_n \circ X, d_\tau + d_X)
=
(\rcom_n \circ X, d_X),
\]
since $d_\tau=0$ by the triviality of the action of the Poisson bracket for objects of the form $j^* X$ obtained by restriction from $\ucom\Mod$. On the other hand, we have
\[
\iota^* (L(X)) = \iota^* (\rcom_n \circ X, d_X)
\]
as well, so that the left-adjoints commute.
\end{proof}
Furthermore, note that all left-adjoints in \eqref{equ:Segaldiagram 1} are strong symmetric monoidal functors, and the right-adjoints are (consequently) lax symmetric monoidal.
We hence get a diagram of categories of commutative algebra objects, such that the diagrams of left- and right-adjoints commute.
\begin{equation}\label{equ:Segaldiagram 2}
\begin{tikzcd}
\CAlg\left(\pois_n\Mod\right) \ar[shift left, dashed]{rr}{L}
\ar[shift left]{dr}{j_*}
&
&
\CAlg\left(\uifrob_n\Mod\right) \ar[shift left]{ll}{R}
\ar[shift left]{dl}{\iota^*}
\\
& \CAlg\left(\ucom\Mod \right)
\ar[shift left, dashed]{ul}{j^*} 
\ar[shift left, dashed]{ur}{\iota_!} &
\end{tikzcd}
\end{equation}

Now we say that an object $A\in \CAlg(\ucom\Mod)$ satisfies the Segal condition (or is of \emph{power type} in the language of \cite{WillwacherObstruction}) if the canonical morphisms 
\[
A(1)^{\otimes r} \to A(r)
\]
are quasi-isomorphisms for all $r$.
We say that an object $B\in \CAlg(\pois_n\Mod)$ satisfies the Segal condition if $j_* B$ satisfies the Segal condition in $\CAlg(\ucom\Mod)$. (In \cite{WillwacherObstruction} such $B$ are called of \emph{configuration space type}.)
Finally, we say that an object $C\in \CAlg(\uifrob_n\Mod)$ satisfies the Segal condition if the morphisms $C(1)^{\otimes r} \to C(r)$ are quasi-isomorphisms, which is equivalent to saying that $\iota^* C$ satisfies the Segal condition in $\CAlg(\ucom\Mod)$.

But in this formulation it is then clear from the commutativity of the diagram of right-adjoints in \eqref{equ:Segaldiagram 2} that $C\in \CAlg(\uifrob_n\Mod)$ satisfies the Segal condition if and only if $R(C)\in \CAlg(\pois_n\Mod)$ satisfies the Segal condition.
Applying this to $C=L(B)$ for some $B\in \CAlg(\pois_n\Mod)$ and noting that $R(L(B))\simeq B$ by Proposition \ref{prop:quillen equivalence} we conclude that also $B$ satisfies the Segal condition if and only if $L(B)$ satisfies the Segal condition.
This finishes the proof of Theorem \ref{thm:main 1}. \hfill\qed

\begin{rem}
Let us point out two further convenient facts:
Let $f:X\to Y$ be a morphism in $\CAlg(\uifrob_n\Mod)$ (or in fact in $\CAlg(\pop)$ for $\pop$ the PROP envelope for a properad $\POp$).
Then if $f$ is a quasi-isomorphism, $X$ satisfies the Segal condition if and only if $Y$ does. Conversely, assume that both $X$ and $Y$ satisfy the Segal condition. Then $f$ is a quasi-isomorphism iff the arity one part $X(1)\to Y(1)$ is a quasi-isomorphism. Both statements follow immediately from looking at the commutative squares 
\[
\begin{tikzcd}
    X(1)^{\otimes r} \ar{r}{f^{\otimes r}} \ar{d} & Y(1)^{\otimes r} \ar{d} \\ 
    X(r) \ar{r}{f} & Y(r)
\end{tikzcd}.
\]
\end{rem}

\section{Homotopy transfer/rigidification result for involutive Frobenius algebras}
\label{sec:htpy transfer}

\begin{thm}\label{thm:htpy transfer}
    Let $M\in \CAlg(\uifrob_n\Mod)$ satisfy the Segal condition.
    Let $\POp$ be a cofibrant replacement of the properad $\uIFrob_n$ and let $\pop$ be its PROP envelope. Define the graded vector space $H:=H(M(1))$.
    Then there is a $\POp$-algebra structure on $H$, i.e., a properad morphism $\POp\to \End_H$, such that the objects $M$ and $H^{\otimes \bullet}$ are quasi-isomorphic in the category $\CAlg(\pop\Mod)$
\end{thm}
Here the object $H^{\otimes \bullet}\in \CAlg(\pop\Mod)$ is defined such that $H^{\otimes \bullet}(r) = H^{\otimes r}$, with the natural commutative product and the $\pop$-action induced from the $\POp$-algebra structure on $H$.

Note that this is a type of homotopy transfer for $\POp$-algebra structures. We follow the strategy of \cite{CamposWillwacher}. Namely, we let $\pBimod$ be the category of pointed properadic bimodules. Objects are triples $(\POp, \MOp, \QOp)$ where $\POp$ and $\QOp$ are properads and $\MOp$ is a $\POp\dash\QOp$-bimodule with respect to the ``connected'' composition product $\boxtimes_c$ (following \cite{Vallette}) together with a map $\Q \to \MOp(1,1)$, or equivalently a distinguished element of $\MOp(1,1)$. In particular, there are induced maps $\POp \to \MOp \from \QOp$ by composition with (multiple copies of) the distinguished element. We call $(\POp,\MOp,\QOp)$ a quasi right torsor if the map $\QOp \to \MOp$ is an equivalence. Given a map of properads $\POp \to \QOp$ we obtain an associated quasi right torsor via the formula $(\POp, \QOp, \QOp)$ with marked element in $\QOp(1,1)$ being the identity.

It follows from \cite{Vallette} that from a properadic bimodule $(\POp, \MOp, \QOp)$ we obtain a lax symmetric monoidal functor
\[
\mop \circ_{q} (\dash) \colon \pop\Mod \to \qop\Mod,
\]
where $\pop$, $\mop$ and $\qop$ are the prop envelopes (that is the free symmetric algebra in bisymmetric sequences) of $\POp$, $\MOp$ and $\QOp$, respectively.

\begin{prop}\label{prop:quasi torsor}
    Let $(\POp, \MOp, \QOp)$ be a quasi right torsor. Then there is a zig zag of equivalences
    \[
    (\POp, \MOp, \QOp) \from (\AOp, \NOp, \BOp) \to (\AOp, \ROp, \ROp).
    \]
\end{prop}
\begin{proof}
    The statement is equivalent to saying that a quasi right torsor is weakly equivalent to an object in the image of the functor $A \colon \Ppd^{\rightarrow} \to \pBimod ; (\POp \to \QOp) \mapsto (\POp, \QOp, \QOp)$ from the arrow category of the category of properads.
    We first show an analogous statement for free objects. Namely, as described above, there is a forgetful functor $\pBimod \to S\dash\bimod^{\leftarrow\rightarrow}$ to the cospan category of bisymmetric sequences sending $(\POp, \MOp, \QOp) \to (\POp \to \MOp \from \QOp)$. Let $L$ denote its left adjoint. In particular, it sends $(a \to a \oplus b \oplus c \from c)$ to $(\FreeP(a), \FreeP(a) \boxtimes_c (b \oplus \mathbb{1}) \boxtimes_c \FreeP(c))$ where $\FreeP(a)$ and $\FreeP(c)$ are the free properads on generators $a$ and $c$, respectively. Note that $S\dash\bimod^{\leftarrow\rightarrow}$ has a particularly explicit cofibrant replacement comonad (in the sense of \cite{Riehl}) given by
    \[
    (x \to \mathrm{cyl}(x \oplus z \to y) \from z) \to (x \to y \from y).
    \]
    We claim that, $L$ sends that morphism to a weak equivalence. Indeed, since taking free properads respects weak equivalences we only need to verify the statement for the bimodule, that is
    \[
    \FreeP(x) \boxtimes_c ( x[1] \oplus y \oplus z[1] \oplus \mathbb{1} ) \boxtimes_c \FreeP(z) \to \FreeP(x) \boxtimes_c (x[1] \oplus \mathbb{1}) \boxtimes_c \FreeP(y) \to \FreeP(y)
    \]
    are both equivalences. Here we omitted some differentials from the notation, more precisely, there is an exhaustive filtration by the number of tensor factors appearing in the middle composition factor (equivalently, the polynomial degree of the functor $V \mapsto \FreeP(x) \boxtimes_c V \boxtimes_c \FreeP(z)$), whose associated graded is the above. From that description we conclude that the first map is an equivalence. For the second equivalence, we filter by the number of tensor factors of $y$, so that the associated graded is the properadic Koszul complex and hence contractible by \cite{Vallette}.

    To obtain the statement for general $(\POp, \MOp, \QOp)$ we take the cotriple resolution with respect to $T(\POp, \MOp, \QOp) = L(\POp \to \mathrm{cyl}(\POp \oplus \QOp \to \MOp) \from \QOp)$. Note that $T$ is a comonad as composition of an adjunction with the comonadic cofibrant replacement functor. The above shows that $T \to A \circ A^L \circ T$ is an equivalence on objects $(x \to y \xleftarrow{\simeq} z)$. We obtain an equivalence of simplicial objects
    \[
    \Omega^\bullet(\Delta^n) \otimes T^n(\POp, \MOp, \QOp) \to A\left( \Omega^\bullet(\Delta^n) \otimes A^L T^n(\POp, \MOp, \QOp) \right).
    \]
    The left hand side (sum-)totalizes to an object equivalent to $(\POp, \MOp, \QOp)$ and the right hand side to an object in the image of $A$.
\end{proof}

\begin{proof}[Proof of Theorem \ref{thm:htpy transfer}]
Choose a map $H \to M(1)$ inducing the identity on cohomology. We then consider the quasi right torsor $(\POp, \Hom(H, M), \End(H))$, where $\Hom(H,M)(i,j) = \Hom(H^{\otimes i}, M(j))$ and $\End(H)(i,j) = \Hom(H^{\otimes i}, H^{\otimes j})$. By Proposition \ref{prop:quasi torsor} we can find a zig zag of equivalences
\[
(\POp, \Hom(H, M), \End(H))
 \xleftarrow[f]{\simeq}
 (\AOp, \MOp, \BOp)
\xrightarrow[g]{\simeq} (\AOp, \ROp, \ROp)
\]
After possibly replacing $\AOp$ and $\BOp$ we can assume that $\AOp \to \POp$ and $\BOp \to \ROp$ are fibrations. Since $\POp$ is cofibrant we can find the following dashed morphisms
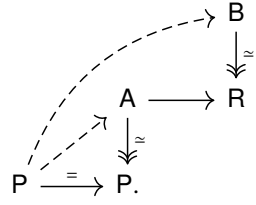
\[
\begin{tikzcd}
    & & \BOp \ar[d, twoheadrightarrow, "\simeq"]  \\
    & \AOp \ar[d, twoheadrightarrow, "\simeq"] \ar[r] & \ROp \\
    \POp \ar[r, "="] \ar[ur, dashed] \ar[uurr, bend left, dashed] & \POp.
\end{tikzcd}
\]
In particular, $H$ becomes a $\POp$-algebra. And we obtain the zig zag of $\POp$-$\BOp$-bimodules $\BOp \to \ROp \xleftarrow{\simeq} \MOp$. We can further assume that $\MOp$ is semi-free as pointed $\BOp$ right module and that $\MOp \to \ROp$ is a fibration, so that taking the pullback is the homotopy pullback and we obtain $\BOp \xleftarrow{\simeq} \UOp \to \MOp$ where we furthermore can assume that $\UOp$ is semi-free as $\BOp$ right module. Passing to envelopes we obtain a zig zag of lax $\POp$-algebras
\[
H^{\otimes \bullet} \xleftarrow{\simeq} \mathsf{u} \circ_{\mathsf{b}} H^{\otimes \bullet} \to \mathsf{m} \circ_{\mathsf{b}} H^{\otimes \bullet} \to \mathsf{hom}(H,M) \circ_{\mathsf{end}(H)} H \to M.
\]

It remains to show that the map $\UOp \circ_{\BOp} H \to M$ is a quasi-isomorphism.
Since $\MOp$ and $\Hom(H,M)$ are pointed so is $\UOp$. Hence, after forgetting the $\POp$-structure we obtain the commuting diagram
\[
\begin{tikzcd}
    & H^{\otimes \bullet} \ar[dl, equal] \ar[dr] \ar[d] & \\
    H^{\otimes \bullet} & \ar[l, "\simeq"] \mathsf{u} \circ_{\mathsf{b}} H^{\otimes \bullet} \ar[r] & M
\end{tikzcd}
\]
where the map $H^{\otimes \bullet} \to M$ is obtained as $H^{\otimes n} \to M(1)^{\otimes n} \to M(n)$ and thus an equivalence by the Segal condition.

\end{proof}

\section{Two models for the properad $\uIFrob_n$}
\label{sec:GX}
\subsection{Cofibrant resolutions for $\uIFrob_{n}$}
\label{sec:cofib uIFrob}
Canonical cofibrant resolutions of augmented properads can usually be obtained by bar-cobar construction. For non-augmented properads like $\uIFrob_n$ the bar-cobar resolution has to be slightly adapted, as is described in \cite{HirshMilles}, to which we refer for details.
We denote $\Bar$ and $\Bar^c$ the bar and cobar construction as described there and define our cofibrant resolution of $\uIFrob_{n}$ as
\[
\uIFrob_{n,\infty} := \Bar^c\Bar \uIFrob_{n} \xrightarrow{\sim} \uIFrob_{n}.
\]
Note that this resolution is not minimal. But for our purposes it will in fact be more economic to work with this canonical resolution.

Let us describe explicitly the properad $\uIFrob_{n,\infty}$. It is quasi-free, generated by the coaugmentation coideal of the bar construction $\overline{\Bar \uIFrob_{n}}$.
Elements of $\Bar \uIFrob_{n}(r,s)$ in turn can be understood as linear combinations of connected directed acyclic graphs with $s$ inputs and $r$ outputs.
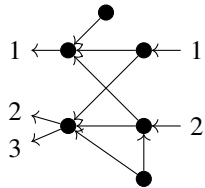
\[
\begin{tikzpicture}
    \node[int] (v1) at (0,.5) {};
    \node[int] (v2) at (0,-.5) {};
    \node[int] (v3) at (1,.5) {};
    \node[int] (v4) at (1,-.5) {};
    \node[int] (v5) at (.5,1) {};
    \node[int] (v6) at (1,-1.2) {};
    \node (l1) at (-.7,.5) {$1$};
    \node (l2) at (-.7,-.3) {$2$};
    \node (l3) at (-.7,-.8) {$3$};
    \node (r1) at (1.7,.5) {$1$};
    \node (r2) at (1.7,-.5) {$2$};
    \draw[->] (r1) edge (v3) (r2) edge (v4) (v3) edge (v1) edge (v2) (v4) edge (v1) edge (v2) 
    (v1) edge (l1) (v2) edge (l2) edge (l3)
    (v5) edge (v1) (v6) edge (v2) edge (v4);
\end{tikzpicture}
\]
Every vertex must have at least one out-edge, but may have zero in-edges. Vertices may have valence 1 or 2. However, there must not be ``passing vertices'', i.e., vertices with one in- and one out-edge.
\[
\begin{tikzpicture}
    \node[int] (v) at (0,0) {};
    \draw[->] (v) edge +(-.7,0); 
    \draw[<-] (v) edge +(.7,0); 
\end{tikzpicture}
\quad\quad\text{(not allowed)}
\]
The differential $d$ on $\uIFrob_{n,\infty}$ is fully determined by its value on the generators $\overline{\Bar \uIFrob_{n}}$.
There the differential has three parts, cf. \cite[section 3.3.6]{HirshMilles}:
\[
d= d_c + d_1 + d_{prop}.
\]
Here $d_c$ acts by contracting one edge of the graph, summing over all choices of edges whose contraction does not produce directed cycles or passing vertices.
The term $d_1$ only acts nontrivally when $r=s=1$ and the graph has the following special form, in which case it is sent to the operadic unit in $\uIFrob_{n,\infty}$:
\[
\begin{tikzpicture}
    \node[int] (v) at (0,0) {};
    \node[int] (w) at (0.5,0.5) {};
    \node (r1) at (1,0) {$1$};
    \node (l1) at (-.7,0) {$1$};
    \draw[->] (r1) edge (v) (v) edge (l1) (w) edge (v);
\end{tikzpicture}
\ \xrightarrow{d_1} \ \mathbf 1.
\]
Note that this may be understood as a special form of contraction of the unique edge.
Finally, the piece $d_{prop}$ sends our element $x$ of $\overline{\Bar \uIFrob_{n}}$ to a sum of properadic compositions of two elements, obtained by applying the reduced cooperadic cocomposition to $x$, see again \cite[section 3.3.6]{HirshMilles}.

\subsection{A version with undirected graphs}
Our next goal is to define a curved coproperad $\GX_n$ weakly equivalent to $\Bar \uIFrob_n$ above, built from undirected graphs instead of directed as above. 

Elements of $\GX_n(r,s)$ (for $r\geq 1, s\geq 0$ are formal $\Q$-linear combinations of pairs $(\Gamma,o)$, where $\Gamma$ is an admissible (non-directed) connected graph with $s$ input- and $r$ output-legs. Admissible means that the graph does not have vertices of valence 2. It may however have vertices of valence 1, self-edges or multiple edges.
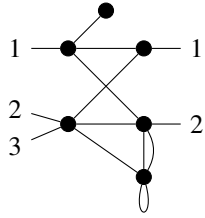
\begin{align*}
\begin{tikzpicture}
    \node[int] (v1) at (0,.5) {};
    \node[int] (v2) at (0,-.5) {};
    \node[int] (v3) at (1,.5) {};
    \node[int] (v4) at (1,-.5) {};
    \node[int] (v5) at (.5,1) {};
    \node[int] (v6) at (1,-1.2) {};
    \node (l1) at (-.7,.5) {$1$};
    \node (l2) at (-.7,-.3) {$2$};
    \node (l3) at (-.7,-.8) {$3$};
    \node (r1) at (1.7,.5) {$1$};
    \node (r2) at (1.7,-.5) {$2$};
    \draw (r1) edge (v3) (r2) edge (v4) (v3) edge (v1) edge (v2) (v4) edge (v1) edge (v2) 
    (v1) edge (l1) (v2) edge (l2) edge (l3)
    (v5) edge (v1) 
    (v6) edge (v2) edge (v4) edge[bend right] (v4) edge[loop below] (v6);
\end{tikzpicture}
\end{align*}
The part $o$ of our pair is an orientation datum, depending on the parity of $n$. For $n$ even the orientation is an ordering on the set of edges and out-legs. For $n$ odd the orientation is an ordering on the set of vertices and half-edges, with the output-legs counting as edges, but not the input legs.
We impose the following relations:
\begin{itemize}
    \item Isomorphism: If $\phi : \Gamma\to \Gamma'$ is a graph isomorphism, then $(\Gamma,o)=(\Gamma',\phi_*o)$, with $\phi_*o$ the natural orientation on $\Gamma'$ induced from $o$ via $\phi$.
    \item Sign: If $o$, $o'$ are orientations (orderings) that differ by some permutation $\sigma$, then we set $(\Gamma,o)=(-1)^\sigma (\Gamma,o')$.
\end{itemize}

We will sometimes be slightly imprecise and talk about elements of $\GX_n$ as linear combinations of graphs $\Gamma$, leaving the orientation data $o$ implicit. The cohomological degree of a pair $(\Gamma,o)$ is 
\[
(\#\text{edges})(n-1) - (\#\text{vertices})n,
\]
where the out-legs also count as edges.

There is a natural structure of a graded coproperad on $\GX_n$.
Since it is easier to work with properads than with coproperads, we will equivalently define the dual properadic composition on $\GX_n^*$. Note that elements of $\GX_n^*$ can naturally also be understood as (series of) oriented graphs, so it is sufficient to define the composition on pairs of graphs. For $(\Gamma,o)\in \GX^*(r,A\sqcup S)$ and $(\Gamma',o')\in \GX^*(A\sqcup S',t)$ we define the properadic composition merely by gluing the input-legs of $\Gamma$ to the like-labeled output legs of $\Gamma'$. There is a natural orientation $o\circ o'$ on the result, by letting the elements of the ordering $o$ stand to the left of those of the ordering $o'$.
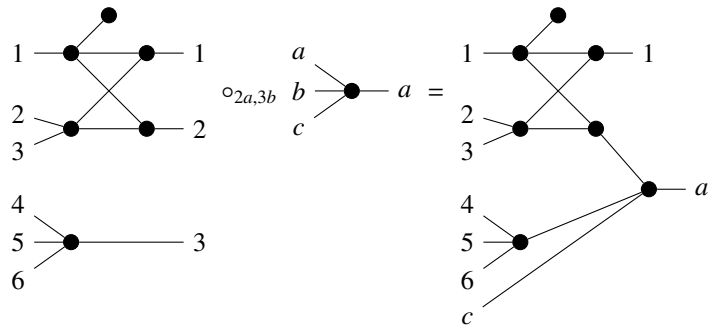
\[
\begin{tikzpicture}
    \node[int] (v1) at (0,.5) {};
    \node[int] (v2) at (0,-.5) {};
    \node[int] (v3) at (1,.5) {};
    \node[int] (v4) at (1,-.5) {};
    \node[int] (v5) at (.5,1) {};
    \node (l1) at (-.7,.5) {$1$};
    \node (l2) at (-.7,-.3) {$2$};
    \node (l3) at (-.7,-.8) {$3$};
    \node (r1) at (1.7,.5) {$1$};
    \node (r2) at (1.7,-.5) {$2$};
    \draw (r1) edge (v3) (r2) edge (v4) (v3) edge (v1) edge (v2) (v4) edge (v1) edge (v2) 
    (v1) edge (l1) (v2) edge (l2) edge (l3)
    (v5) edge (v1);
    \node (l4) at (-.7,-1.5) {$4$};
    \node (l5) at (-.7,-2) {$5$};
    \node (l6) at (-.7,-2.5) {$6$};
    \node (r3) at (1.7,-2) {$3$};
    \node[int] (w1) at (0,-2) {};
    \draw (w1) edge (l4) edge (l5) edge (l6) edge (r3);
\end{tikzpicture}
\circ_{2a,3b}
\begin{tikzpicture}
    \node (la) at (-.7,.5) {$a$};
    \node (lb) at (-.7,0) {$b$};
    \node (lc) at (-.7,-.5) {$c$};
    \node (ra) at (.7,0) {$a$};
    \node[int] (w1) at (0,0) {};
    \draw (w1) edge (la) edge (lb) edge (lc) edge (ra);
\end{tikzpicture}
=
\begin{tikzpicture}
    \node[int] (v1) at (0,.5) {};
    \node[int] (v2) at (0,-.5) {};
    \node[int] (v3) at (1,.5) {};
    \node[int] (v4) at (1,-.5) {};
    \node[int] (v5) at (.5,1) {};
    \node (l1) at (-.7,.5) {$1$};
    \node (l2) at (-.7,-.3) {$2$};
    \node (l3) at (-.7,-.8) {$3$};
    \node (r1) at (1.7,.5) {$1$};
    \draw (r1) edge (v3) (v3) edge (v1) edge (v2) (v4) edge (v1) edge (v2) 
    (v1) edge (l1) (v2) edge (l2) edge (l3)
    (v5) edge (v1);
    \node (l4) at (-.7,-1.5) {$4$};
    \node (l5) at (-.7,-2) {$5$};
    \node (l6) at (-.7,-2.5) {$6$};
    \node[int] (w1) at (0,-2) {};
    \draw (w1) edge (l4) edge (l5) edge (l6);
    \node (lc) at (-.7,-3) {$c$};
    \node (ra) at (2.4,-1.3) {$a$};
    \node[int] (w2) at (1.7,-1.3) {};
    \draw (w2) edge (v4)  edge (w1) edge (ra) edge (lc);
\end{tikzpicture}
\]
The coproperadic cocomposition on $\GX$ is defined by the dual operation of cutting graphs into two pieces.

Finally we endow $\GX_n$ with the structure of a curved coproperad. To this end we have to define a derivation $d$ and curvature $\theta\in \GX_n^*$ such that $d^2$ is given by the action of $\theta$, see \cite{HirshMilles}.
In our case the action of $d$ is simply by summing over all ways of contracting one (non-leg) edge,
\[
d(\Gamma,o) = \sum_{e} \epsilon_e (\Gamma/e, o_e),
\]
with the orientation $o_e$ and sign $\epsilon_e$ depending on the parity of $n$ as follows:
\begin{itemize}
\item If $n$ is even then $o$ is an ordering of the set of edges, with $e$ appearing in position $j_e\in\{1,2,\dots\}$. Then $o_e$ is the ordering obtained by removing $e$ from $o$ and $\epsilon_e = (-1)^{j_e-1}$.
\item If $n$ is odd then $o$ is an ordering of the set of vertices and half-edges. Equivalently, we can order the vertices and orient the edges. If our edge $e$ points from the first to the second vertex, then $\epsilon_e=+1$ and $o_e$ is obtained by omitting $e$ and merging the adjacent vertices in $o$. For other edges the sign is determined by equivariance.
\end{itemize} 

The curvature is given by the graph 
\[
\theta = \begin{tikzpicture}
    \node[int] (v) at (0,0) {};
    \node[int] (w) at (0.5,0.5) {};
    \node (r1) at (1,0) {$1$};
    \node (l1) at (-.7,0) {$1$};
    \draw (r1) edge (v) (v) edge (l1) (w) edge (v);
\end{tikzpicture}.
\]

\begin{lemma}
    The above data equip $\GX_n$ with a well-defined structure of a curved coproperad.
\end{lemma}
\begin{proof}
We check the equivalent dual statement that $\GX_n^*$ is a well-defined curved properad. It is obvious that the operadic composition by gluing graphs as above is associative. It is also obviously compatible with the differential (vertex splitting) since the composition does not alter the number of vertices, nor the set of half-edges attached to a vertex.

Next consider the double application $\delta^2\Gamma$ of the differential to a graph $\Gamma$.
If we allowed vertices of all valence, i.e., if we had not excluded vertices of valence 2, then it is a well-known verification that $\delta^2\Gamma=0$. 
Hence the terms of $\delta^2\Gamma$ are those formally obtained by double splits, with the first split creating a bivalent vertex and the second split removing the bivalent vertex by creating a univalent vertex. Let us distinguish two cases, according to whether the new bivalent vertex connects to one vertex and a leg, or to two vertices:
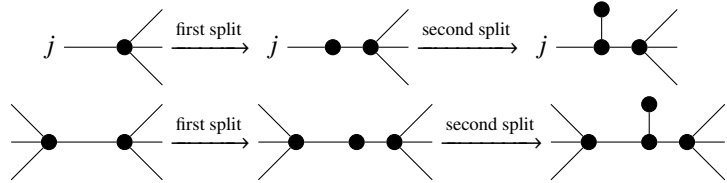
\begin{align*}
\begin{tikzpicture}
    \node[int] (v) at (0,0) {};
    \node (l1) at (-1,0) {$j$};
    \draw (l1) edge (v) (v) edge +(.5,0) edge +(.5,-.5) edge +(.5,.5);
\end{tikzpicture}
&\xrightarrow{\text{first split}}
\begin{tikzpicture}
    \node[int] (v) at (0,0) {};
    \node[int] (w) at (-.5,0) {};
    \node (l1) at (-1.3,0) {$j$};
    \draw (l1) edge (w) (v) edge +(.5,0) edge +(.5,-.5) edge +(.5,.5) edge (w);
\end{tikzpicture}
\xrightarrow{\text{second split}}
\begin{tikzpicture}
    \node[int] (v) at (0,0) {};
    \node[int] (w) at (-.5,0) {};
    \node[int] (u) at (-.5,.5) {};
    \node (l1) at (-1.3,0) {$j$};
    \draw (l1) edge (w) (v) edge +(.5,0) edge +(.5,-.5) edge +(.5,.5) edge (w) (w) edge (u);
\end{tikzpicture}
\\
\begin{tikzpicture}
    \node[int] (v) at (0,0) {};
    \node[int] (l1) at (-1,0) {};
    \draw (l1) edge (v) (v) edge +(.5,0) edge +(.5,-.5) edge +(.5,.5);
    \draw (l1) edge +(-.5,0) edge +(-.5,.5) edge +(-.5,-.5);
\end{tikzpicture}
&\xrightarrow{\text{first split}}
\begin{tikzpicture}
    \node[int] (v) at (0,0) {};
    \node[int] (w) at (-.5,0) {};
    \node[int] (l1) at (-1.3,0) {};
    \draw (l1) edge (w) (v) edge +(.5,0) edge +(.5,-.5) edge +(.5,.5) edge (w);
    \draw (l1) edge +(-.5,0) edge +(-.5,.5) edge +(-.5,-.5);
\end{tikzpicture}
\xrightarrow{\text{second split}}
\begin{tikzpicture}
    \node[int] (v) at (0,0) {};
    \node[int] (w) at (-.5,0) {};
    \node[int] (u) at (-.5,.5) {};
    \node[int] (l1) at (-1.3,0) {};
    \draw (l1) edge (w) (v) edge +(.5,0) edge +(.5,-.5) edge +(.5,.5) edge (w) (w) edge (u);
    \draw (l1) edge +(-.5,0) edge +(-.5,.5) edge +(-.5,-.5);
\end{tikzpicture}
\end{align*}
In the second case the shown terms are in fact produced twice for each edge, once from splitting the left-hand vertex and once from the right-hand vertex. Both contributions come with opposite sign and hence cancel.
There remain only the first contribution, and they are precisely the action of the graph $\theta$ above on $\Gamma$, so that 
\[
\delta^2\Gamma=[\theta,\Gamma]
\]
as desired.
\end{proof}

\subsection{Weak equivalence}
Our next goal is to define a weak equivalence (in a precise sense) between the curved coproperads $\GX_n$ and $\Bar \uIFrob_{n}$.
More precisely, a morphism of curved coproperads 
\[
\Theta: \GX_n \to \Bar \uIFrob_{n}
\]
is a morphism of graded coproperads such that $\Theta$ intertwines the differentials and the curvatures, $d\Theta=\Theta d$, $\theta_{\uIFrob_{n}} \Theta=\theta_{\GX_n}$.
Note that it does not make sense to ask that $\Theta$ is a quasi-isomorphism since the differentials on either side do not square to zero. 
However, we note that both $\GX_n$ and $\Bar \uIFrob_{n}$ carry a filtration by the number of univalent vertices, and the associated graded coproperads $\gr\GX_n$ and $\gr\Bar \uIFrob_{n}$ are dg coproperads.\footnote{In fact, $\gr\Bar \uIFrob_{n}=\Bar q\uIFrob_{n}$ is the bar construction of the associated quadratic properad $q\uIFrob_{n}$, see \cite{HirshMilles}.} Our morphism $\Theta$ will respect the number of univalent vertices on either side, and induce a quasi-isomorphism on the associated graded dg coproperads. 

The idea of the definition of $\Theta$ and the verification of its properties follows closely the discussion of a similar map $\Phi$ between the hairy graph complex and its directed acyclic variant in \cite{AssarMarko}. For some verifications that are essentially identical as in loc. cit. we will refer to the arguments given there.

To define $\Theta$, we in fact switch to the dual properads again and define instead the dual morphism 
\[
\Theta^* : (\Bar \uIFrob_{n})^* \to \GX_n^*.
\]
Let $(\Gamma,o)\in (\Bar \uIFrob_{n})^*(r,s)$ be a directed acyclic graph. Then we say that $\Gamma$ is \emph{good} if the following holds:
\begin{itemize}
    \item Every vertex of $\Gamma$ has either exactly one out-edge or it has two out-edges and zero in-edges. 
\end{itemize}
We call vertices of the second type (i.e., $
\begin{tikzpicture}
    \node[int](v) at(0,0) {};
    \draw[->] (v) edge +(-.5,0) edge +(.5,0);
\end{tikzpicture}$) \emph{inessential} and all other vertices \emph{essential}.
Now we define:
\begin{itemize}
\item If $\Gamma$ is not good, then we set $\Theta^*(\Gamma)=0$.
\item If $\Gamma$ is good, then we form a graph $\Gamma'$ from $\Gamma$ by replacing each inessential vertex with its two adjacent edges by one edge, and all remaining directed edges by undirected edges, for example:
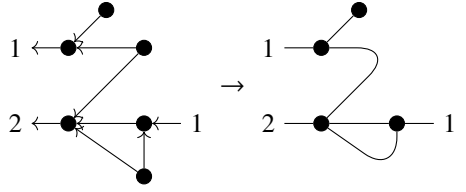
\[
\begin{tikzpicture}
    \node[int] (v1) at (0,.5) {};
    \node[int] (v2) at (0,-.5) {};
    \node[int] (v3) at (1,.5) {};
    \node[int] (v4) at (1,-.5) {};
    \node[int] (v5) at (.5,1) {};
    \node[int] (v6) at (1,-1.2) {};
    \node (l1) at (-.7,.5) {$1$};
    \node (l2) at (-.7,-.5) {$2$};
    \node (r2) at (1.7,-.5) {$1$};
    \draw[->]  (r2) edge (v4) (v3) edge (v1) edge (v2) (v4) edge (v2) 
    (v1) edge (l1) (v2) edge (l2)
    (v5) edge (v1) (v6) edge (v2) edge (v4);
\end{tikzpicture}
\to 
\begin{tikzpicture}
    \node[int] (v1) at (0,.5) {};
    \node[int] (v2) at (0,-.5) {};
    \node (v3) at (1,.5) {};
    \node[int] (v4) at (1,-.5) {};
    \node[int] (v5) at (.5,1) {};
    \node (v6) at (1,-1.2) {};
    \node (l1) at (-.7,.5) {$1$};
    \node (l2) at (-.7,-.5) {$2$};
    \node (r2) at (1.7,-.5) {$1$};
    \draw  (r2) edge (v4) (v4) edge (v2) 
    (v1) edge (l1) (v2) edge (l2)
    (v5) edge (v1);
    \draw[rounded corners=15pt](v2) -- (v6.center) -- (v4);
    \draw[rounded corners=15pt](v1) -- (v3.center) -- (v2);
\end{tikzpicture}.
\]

\item We define an orientation $o'$ on $\Gamma'$ as follows, depending on the parity of $n$:
\begin{itemize}
    \item If $n$ is even then the orientation $o$ is an ordering of the set of vertices and half-edges of $\Gamma$, whereas the orientation $o'$ (to be defined) is an ordering of the set of edges. Now each essential vertex has a unique outgoing half-edge, and we may assume that both occur adjacent to each other in the ordering, the vertex before the half-edge. (The pair together is then even for sign purposes.) Each inessential vertex $v$ comes with 4 half-edges $h_1,h_1',h_2,h_2'$ of the two adjacent edges $(h_j,h_j')$ and we assume that they are adjacent in $o$ and of fixed order $vh_1h_1'h_2h_2'$. The five items together can be thought of as one odd item in $o$. These items together with the remaining half-edges are now in 1-1 correspondence to the edges of $\Gamma'$, and we may just retain their relative order to make $o'$. 
    \item If $n$ is odd then conversely $o$ is an order on the edges of $\Gamma$ and $o'$ (to be defined) an order on the set of vertices and half-edges of $\Gamma'$. We build $o'$ from $o$ as follows: Let $e_1,e_2$ be the two edges adjacent to an inessential vertex that become the edge $(h,h')$ of $\Gamma'$.
    Then we just replace $e_1,e_2$ by $h,h'$ in the ordering.
    Let $e$ be another edge of $\Gamma$, and $v$ be its source vertex. Then we replace $e$ in the ordering by $vhh'$, with $h,h'$ the half-edges of the corresponding edge in $\Gamma'$.
\end{itemize}
\item Finally, we set 
\[
\Theta^*(\Gamma,o):=(\Gamma',o').
\]
\end{itemize}

\begin{lemma}
    The map $\Theta^* : (\Bar \uIFrob_{n})^* \to \GX_n^*$ is a morphism of curved properads. It respects the gradings by loop order and the filtrations by the number of univalent vertices on both sides.
\end{lemma}
\begin{proof}
    It is clear that $\Theta^*$ respects the properadic compositions, the loop order, the number of univalent vertices, and that the curvatures are sent to each other.
    The verification that it intertwines the differentials, i.e., $\delta\Theta^*=\Theta^*\delta$, can be shown in the same way as in \cite[proof of Proposition 13]{AssarMarko}.
\end{proof}

\begin{prop}\label{prop:Theta qiso}
    The map $\Theta^* :\gr(\Bar \uIFrob_{n})^* \to \gr\GX_n^*$ induced by $\Theta^*$ on the associated graded objects with respect to the filtration by the number of univalent vertices is a quasi-isomorphism of dg properads.
\end{prop}
The proof is similar to that of \cite[Proposition 19]{AssarMarko}.
\begin{proof}
    Note that the morphism $\Theta^*$ sends a good graph with $N$ essential vertices to a graph with $N$ vertices. Hence $\Theta^*$ is compatible with the filtration on $\gr(\Bar \uIFrob_{n})^*$ by the number of essential vertices and the filtration on $\gr\GX_n^*$ by the number of vertices.
    We will in fact show that the associated graded map 
    \[
    \Theta^* :\gr\gr(\Bar \uIFrob_{n})^* \to \gr\gr\GX_n^*
    \]
    is a quasi-isomorphism. In fact, we have that $\gr\gr\GX_n^*\cong (\GX_n^*,0)$ is equipped with the zero differential, since the differential always creates one vertex.
    Similarly, $\gr\gr(\Bar \uIFrob_{n})^*\cong ((\Bar \uIFrob_{n})^*, d')$ with $d'$ creating an inessential vertex from an edge:
    \[
    d' : 
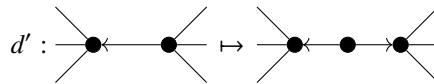
    \begin{tikzpicture}
    \node[int] (v) at (0,0) {};
    \node[int] (l1) at (-1,0) {};
    \draw (l1) edge[<-] (v) 
    (v) edge +(.5,0) edge +(.5,-.5) edge +(.5,.5);
    \draw (l1) edge +(-.5,0) edge +(-.5,.5) edge +(-.5,-.5);
\end{tikzpicture}
\, \mapsto \, 
\begin{tikzpicture}
    \node[int] (v) at (0,0) {};
    \node[int] (w) at (-.7,0) {};
    \node[int] (l1) at (-1.4,0) {};
    \draw (l1) edge[<-] (w) (v) edge +(.5,0) edge +(.5,-.5) edge +(.5,.5) edge[<-] (w);
    \draw (l1) edge +(-.5,0) edge +(-.5,.5) edge +(-.5,-.5);
\end{tikzpicture}
    \]
    Following \cite{AssarMarko} we will furthermore draw directed graphs in $(\Bar \uIFrob_{n})^*$ differently, by replacing an inessential vertex with its two adjacent edges by a marked edge:
    
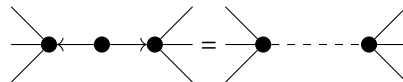
\[
    \begin{tikzpicture}
    \node[int] (v) at (0,0) {};
    \node[int] (w) at (-.7,0) {};
    \node[int] (l1) at (-1.4,0) {};
    \draw (l1) edge[<-] (w) (v) edge +(.5,0) edge +(.5,-.5) edge +(.5,.5) edge[<-] (w);
    \draw (l1) edge +(-.5,0) edge +(-.5,.5) edge +(-.5,-.5);
\end{tikzpicture}
=
\begin{tikzpicture}
    \node[int] (v) at (0,0) {};
    \node[int] (l1) at (-1.4,0) {};
    \draw (l1) edge[dashed] (v) (v) edge +(.5,0) edge +(.5,-.5) edge +(.5,.5);
    \draw (l1) edge +(-.5,0) edge +(-.5,.5) edge +(-.5,-.5);
\end{tikzpicture}
    \]
    Then we will say that the core of a graph $\Gamma$ is the graph obtained by replacing all directed and marked edges by undirected edges. In this language the map $\Phi^*$ is just the projection of graphs to their core, setting graphs that are not good to zero. 
    Furthermore, we have the direct sum decomposition of dg vector spaces  
    \begin{align*}
    \gr\gr(\Bar \uIFrob_{n})^*(r,s) &\cong \bigoplus_{\gamma} B_\gamma,
    \end{align*}
    where the sum is over all possible core graphs $\gamma$, and $B_\gamma$ is the subcomplex spanned by all graphs with core $\gamma$. Recall that the differential on $B_\gamma$ just replaces a directed edge by a marked edge
    
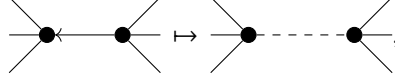
\[
    \begin{tikzpicture}
    \node[int] (v) at (0,0) {};
    \node[int] (l1) at (-1,0) {};
    \draw (l1) edge[<-] (v) 
    (v) edge +(.5,0) edge +(.5,-.5) edge +(.5,.5);
    \draw (l1) edge +(-.5,0) edge +(-.5,.5) edge +(-.5,-.5);
\end{tikzpicture}
\, \mapsto \,
\begin{tikzpicture}
    \node[int] (v) at (0,0) {};
    \node[int] (l1) at (-1.4,0) {};
    \draw (l1) edge[dashed] (v) (v) edge +(.5,0) edge +(.5,-.5) edge +(.5,.5);
    \draw (l1) edge +(-.5,0) edge +(-.5,.5) edge +(-.5,-.5);
\end{tikzpicture}
,
    \]
    sending the graph to zero if this produces a vertex without output edges.
    In particular, this operation does not alter the number of edges in the graph. Hence, we may temporarily fix an ordering of the edges of the core $\gamma$, and assume the edges of our graphs are numbered. More precisely, we do this by temporarily passing to a complex $B_\gamma'$ whose elements are linear combinations of graphs as before but equipped with an isomorphism of their core to $\gamma$, so that $B_\gamma = B_\gamma'/\mathrm{Aut}(\gamma)$.
    
    It then suffices to check that the morphism
    \[
    H(B_\gamma',d')\to \Q\gamma,
    \]
    sending a good graph to its core and other graphs to zero is an isomorphism.

    We next fix a good graph $\Gamma_0$ whose core is $\gamma$.
    We show that the cohomology $H(B_\gamma',d')\cong \Q\Gamma_0$ is one-dimensional, represented by $\Gamma_0$.
    Let $V_j$ be the set of vertices of $\Gamma_0$ that are adjacent to edges $1,\dots,j$ of $\Gamma_0$.
    We may assume w.l.o.g. that our chosen order on the edges of $\Gamma_0$ has the following additional property:
    \begin{itemize}
        \item If edge $j+1$ is between vertices $x$ and $y$, then either the edge is directed from $x$ to $y$ in $\Gamma_0$ and $y\in V_j$, or both $x$ and $y$ are in $V_j$ and the edge is marked. 
    \end{itemize}
    In particular, edge 1 must be an out-leg, edge 2 either another out-leg or an edge adjacent to edge 1 etc.
    We then define a tower of subcomplexes
    \begin{equation}\label{equ:B tower}
    \Q\Gamma_0 = B_{\gamma,N}'\hookrightarrow \cdots \hookrightarrow  B_{\gamma,j}'\hookrightarrow  B_{\gamma,j-1}'\hookrightarrow \cdots \hookrightarrow B_{\gamma,0}' = B_{\gamma}'.
    \end{equation}
    such that $B_{\gamma,j}'\subset B_{\gamma}'$ is spanned by graphs $\Gamma\in B_{\gamma}'$ with the following properties:
    \begin{enumerate}[(i)]
        \item Each edge $1,\dots,j$ of $\Gamma$ has the same type as the corresponding edge of $\Gamma_0$.
        \item Each vertex of $V_j$ has exactly one out-edge in $\Gamma$. 
    \end{enumerate}
    Note that the unique out-edge of any vertex in $V_j$ must be in the set of edges $1,\dots,j$.
    It is clear that in order for the composition \eqref{equ:B tower} to be a quasi-isomorphism, it suffices to check that each map 
    \begin{equation}\label{equ:B incl}
    B_{\gamma,j}'\hookrightarrow  B_{\gamma,j-1}'
    \end{equation}
    is a quasi-isomorphism. 
    Suppose first that the edge $j$ is marked in $\Gamma_0$, connecting vertices $x$ and $y$. 
    Then by the assumption on our edge ordering we have that $V_{j}=V_{j-1}$. Hence any graph $\Gamma\in B_{\gamma,j-1}'$ must be such that edge $j$ in $\Gamma$ is also marked, since vertices $x$ and $y$ must each have a unique out-edge by the condition (ii) above, and these must be of index $\leq j-1$ as we observed above.
    Hence we conclude that $B_{\gamma,j}'=B_{\gamma,j-1}'$, so that the inclusion is trivially a quasi-isomorphism. 
    
    So we are left with the case that edge $j$ is directed in $\Gamma_0$ from vertex $x$ (not in $V_{j-1}$) to vertex $y$ (in $V_{j-1}$).
    We then decompose the complex into the following subspaces
    \[
    B_{\gamma,j-1}'
    \cong
    \begin{tikzcd}[baseline=.65em, column sep = 2mm]
        B_{\gamma,j}' \ar[loop above]{} & \oplus &  A \ar[bend right]{ll}\ar[loop above]{}\ar[bend left]{rr}{f} &\oplus & B \ar[loop above]{}\\
    \end{tikzcd}
    \]
    with the arrows indicating the pieces of the differential. Here $B$ is spanned by graphs for which edge $j$ is marked, and $A$ is spanned by all remaining graphs. (I.e., these are graphs such that edge $j$ is directed from $x$ to $y$, but there are other edges outgoing from $x$.) 
    The arrow $f$ acts by changing the type of edge $j$ from directed to marked.
    Now $f:A\to B$ is an isomorphism, with the inverse map changing the type of edge $j$ from marked to directed from $x$ to $y$. Hence from \cite[Lemma 2.1]{PayneWillwacher21} we readily conclude that the inclusion \eqref{equ:B incl} is a quasi-isomorphism, and hence the proposition follows.
\end{proof}

\begin{cor}\label{cor:Theta qiso properad}
The morphism of properads $\Bar^c\Theta:\Bar^c\GX_n \to \Bar^c\Bar \uIFrob_{n}=\uIFrob_{n,\infty}$ induced by $\Theta$ is a quasi-isomorphism.
\end{cor}
\begin{proof}
We may equip both sides of the morphism $\Bar^c\Theta:\Bar^c\GX_n \to \Bar^c\Bar \uIFrob_{n}=\uIFrob_{n,\infty}$ with the bounded below exhaustive increasing filtrations by the total number of univalent vertices. 
It then suffices to check that the induced morphism on the associated graded properads is a quasi-isomorphism.
To do that note that $\Bar^c(-)$ is always a free properad and has a descending filtration by the number of generators.
In our case this filtration is bounded in each loop order separately.
Hence it suffices to check that the morphism on the associated graded properads is a quasi-isomorphism, which is guaranteed by Proposition \ref{prop:Theta qiso} and the Künneth theorem. 
\end{proof}

\section{The Hoffbeck-Leray-Vallette category}
\label{sec:HLV}
In this section we recall the construction of \cite{HoffbeckLerayVallette} of a simplicial category encoding the homotopy theory of properad algebras, when the properad is the cobar construction of a coproperad. In fact, we shall need a slight variation, allowing the coproperad to be curved, reflecting the fact that $\uIFrob_n$ is not augmented.
This is however not a major change, given that the Koszul duality theory for non-augmented properads and curved coproperads is well-developed \cite{HirshMilles}.
We focus solely on the two cases of the properad being either $\uIFrob_{n,\infty}=\Bar^c \Bar \uIFrob_n$ or $\Bar^c\GX_n$, both being quasi-isomorphic to $\uIFrob_n$.

\subsection{Algebras over $\uIFrob_{n,\infty}$ and $\Bar^c\GX_n$}
Let $V$ be a dg vector space. Then a $\uIFrob_{n,\infty}$-structure on $V$ is the same data as a properad map into the endomorphism properad $\End(V)$
\[
\uIFrob_{n,\infty}\to \End(V).
\]
Note that this morphism does not encode the differential on $V$, which we assume is given from the start. If we want to encode the differential as well into the structure, then we may use the properadic "+"-construction, see \cite[Section 3.1]{MWDef}. To any properad $\POp$ there is a properad $\POp^+$ such that a $\POp^+$-algebra structure on a dg vector space $(V,d)$ is the same data as a pair consisting of (i) a deformation of the differential $d$ on $V$ to a new differential $d+D$ and (ii) a $\POp$-algebra structure on $(V, d+D)$. Concretely, $\POp^+$ is obtained from $\POp$ by freely adjoining a single generator corresponding to $D$, and modifying the differential accordingly.
A $\uIFrob_{n,\infty}$-structure on $V$, together with a deformation of the differential is then encoded by a properad morphism 
\[
\uIFrob_{n,\infty}^+\to \End(V).
\]
If one encodes the differential as part of this properad morphism, then one typically assumes that $V$ is just a graded vector space.

It is well known that properad morphisms $\uIFrob_{n,\infty}^+\to \End(V)$ are the same data as a Maurer-Cartan element in the deformation complex (or convolution algebra)
\begin{equation}\label{equ:Def plus def}
\Def^+(\uIFrob_{n,\infty}, \End(V))
:= \Conv(\Bar \uIFrob_{n}, \End(V))
\cong 
\prod_{r}
\Hom_{S_r}((\Bar \uIFrob_{n}\circ V)(r), V^{\otimes r}).
\end{equation}
Similarly, properad morphisms $\uIFrob_{n,\infty}\to \End(V)$ are the same data as Maurer-Cartan elements in the subcomplex
\[
\Def(\uIFrob_{n,\infty}, \End(V))
\cong 
\prod_{r}
\Hom_{S_r}((\overline{\Bar \uIFrob}_{n}\circ V)(r), V^{\otimes r}).
\]
We refer to \cite{MV-properad} for details. Just note that here the object above is a priori a \emph{curved} dg Lie algebra because $\Bar \uIFrob_{n}$ is a curved coproperad. Having non-trivial curvature reflects the fact that 0 is not a Maurer-Cartan element. The underlying reason is clear here: A non-zero unital Frobenius algebra must have a non-zero multiplication to satisfy the axioms.

Let us make the combinatorial form of the curved dg Lie algebra $\Def^+(\uIFrob_{n,\infty}, \End(V))$ explicit for finite-dimensional $V$.
Elements can be seen as series of directed acyclic graphs whose input legs are decorated by elements of $V^*$ and whose output legs are decorated by elements of $V$.

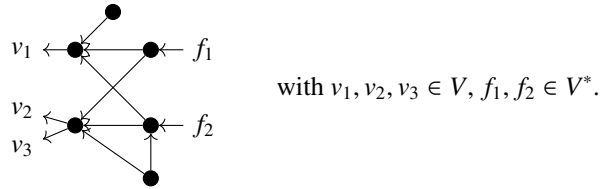
\[
\begin{tikzpicture}
    \node[int] (v1) at (0,.5) {};
    \node[int] (v2) at (0,-.5) {};
    \node[int] (v3) at (1,.5) {};
    \node[int] (v4) at (1,-.5) {};
    \node[int] (v5) at (.5,1) {};
    \node[int] (v6) at (1,-1.2) {};
    \node (l1) at (-.7,.5) {$v_1$};
    \node (l2) at (-.7,-.3) {$v_2$};
    \node (l3) at (-.7,-.8) {$v_3$};
    \node (r1) at (1.7,.5) {$f_1$};
    \node (r2) at (1.7,-.5) {$f_2$};
    \draw[->] (r1) edge (v3) (r2) edge (v4) (v3) edge (v1) edge (v2) (v4) edge (v1) edge (v2) 
    (v1) edge (l1) (v2) edge (l2) edge (l3)
    (v5) edge (v1) (v6) edge (v2) edge (v4);
\end{tikzpicture}
\quad\quad 
\text{with $v_1,v_2,v_3\in V$, $f_1,f_2\in V^*$.}
\]
Still each vertex must have at least one out-edge, but may have 0 input edges, and passing vertices are not allowed. The graph with no vertices and just one edge is allowed, reflecting our passage to $\uIFrob_{n,\infty}^+$ above. 
\[
\begin{tikzpicture}
    \node (l1) at (-.7,.5) {$v_1$};
    \node (r1) at (1.2,.5) {$f_1$};
    \draw (l1) edge (r1);
\end{tikzpicture}
\]
It encodes the deformation of the differential and is not present in $\Def(\uIFrob_{n,\infty}, \End(V))$.
The differential $\delta$ is obtained by splitting vertices, dual to the edge contraction differential above. 
The Lie bracket has the form 
\begin{equation}\label{equ:conv lie def}
[\Gamma_1,\Gamma_2] = \Gamma_1 \bullet \Gamma_2 - (-1)^{|\Gamma_1||\Gamma_2|} \Gamma_2\bullet \Gamma_1,
\end{equation}
with the operation $\Gamma_1\bullet \Gamma_2$ concatenating both graphs, summing over all partial matchings of the inputs of $\Gamma_1$ to the outputs of $\Gamma_2$.
\begin{equation}\label{equ:conv lie def 2}
\Gamma_1\bullet \Gamma_2
=
\sum
\,
\begin{tikzpicture}
    \node[ext, inner sep=1mm] (v) at (0,0) {$\Gamma_1$};
    \node[ext, inner sep=1mm] (w) at (1.5,0) {$\Gamma_2$};
    \draw[<-] (v) edge (w) edge[bend left] (w) edge[bend right] (w) edge +(.5, .7) edge +(.5, -.7)
    (w) edge +(.7,.5) edge +(.7,0) edge +(.7,-.5) ;
    \draw[->] (v) edge +(-.7,.5) edge +(-.7,0) edge +(-.7,-.5)
    (w) edge +(-.5, .7) edge +(-.5, -.7);
\end{tikzpicture}.
\end{equation}
The curvature is the element 
\[
\mu = 
\sum_i
\begin{tikzpicture}
    \node[int] (v) at (0,0) {};
    \node[int] (w) at (0.5,0.5) {};
    \node (r1) at (1,0) {$e_i^*$};
    \node (l1) at (-.7,0) {$e_i$};
    \draw[->] (r1) edge (v) (v) edge (l1) (w) edge (v);
\end{tikzpicture}
\]
where $e_i$ runs over a basis of $V$ and $e_i^*$ is the dual basis. 

For the graphical description above we used the finite-dimensionality of $V$ in so far that we took the decoration of the input legs in $(V^*)^{\otimes k}$, so each of the $k$ input legs carries one decoration in $V^*$. For infinite-dimensional $V$ we need to resort to formula \eqref{equ:Def plus def} instead, i.e., the previous decoration by $(V^*)^{\otimes k}$ should be replaced by $\Hom(V^{\otimes k},-)$. Nevertheless, our formula for the differential and Lie bracket still works, with minor adaptations, see e.g. \cite[Section 2.3]{HLVold}.

The above construction is obviously functorial in the curved coproperad used. In particular, we may replace the role of $\Bar\uIFrob_n$ by the curved coproperad $\GX_n$. Thus we obtain an explicit description of algebras over the properad $\Bar^c \GX_n$. Concretely, the $(\Bar^c \GX_n)^+$-algebra structures on the dg vector space $V$ are in one-to-one correspondence to Maurer-Cartan elements in the curved dg Lie algebra 
\[
\Def^+(\Bar^c \GX_n, \End(V))
:= \Conv(\GX_{n}, \End(V))
\cong 
\prod_{r}
\Hom_{S_r}((\GX_n\circ V)(r), V^{\otimes r}).
\]
For finite dimensional $V$, elements of the latter curved dg Lie algebra can be understood as (possibly infinite) linear combinations of undirected graphs with inputs decorated by elements of $V^*$ and outputs decorated by elements of $V$:
\[
\begin{tikzpicture}
    \node[int] (v1) at (0,.5) {};
    \node[int] (v2) at (0,-.5) {};
    \node[int] (v3) at (1,.5) {};
    \node[int] (v4) at (1,-.5) {};
    \node[int] (v5) at (.5,1) {};
    \node (l1) at (-.7,.5) {$v_1$};
    \node (l2) at (-.7,-.3) {$v_2$};
    \node (l3) at (-.7,-.8) {$v_3$};
    \node (r1) at (1.7,.5) {$f_1$};
    \node (r2) at (1.7,-.5) {$f_2$};
    \draw (r1) edge (v3) (r2) edge (v4) (v3) edge (v1) edge (v2) (v4) edge (v1) edge (v2) 
    (v1) edge (l1) (v2) edge (l2) edge (l3)
    (v5) edge (v1);
\end{tikzpicture}
\quad\quad 
\text{with $v_1,v_2,v_3\in V$, $f_1,f_2\in V^*$.}
\]
The graphs may have univalent vertices, but no bivalent vertices, in accordance to the definition of $\GX_n$. The differential is given by vertex splitting. The Lie bracket is defined analogously to 
\eqref{equ:conv lie def}, \eqref{equ:conv lie def 2} above, just using undirected graphs.
The curvature is the element
\[
\mu = 
\sum_i
\begin{tikzpicture}
    \node[int] (v) at (0,0) {};
    \node[int] (w) at (0.5,0.5) {};
    \node (r1) at (1,0) {$e_i^*$};
    \node (l1) at (-.7,0) {$e_i$};
    \draw (r1) edge (v) (v) edge (l1) (w) edge (v);
\end{tikzpicture},
\]
where the $e_i$ form a a basis of $V$ and $e_i^*$ are the dual basis elements of $V^*$.

By functoriality of the construction we obtain from the morphism of curved coproperads $\GX_n\to \Bar\uIFrob_n$ a morphism of curved dg Lie algebras 
\begin{equation}\label{equ:def map}
\Def^+(\uIFrob_{n,\infty}, \End(V))\to \Def^+(\Bar^c \GX_n, \End(V)).
\end{equation}

\subsubsection{Alternative definition of $\uIFrob_{n,\infty}$-structures}
Several different equivalent encodings of the data of a $\uIFrob_{n,\infty}$-structure on a vector space $V$ habe been studied in \cite{HLVold}. We recall here one further such encoding, that will be helpful later.
To this end, let
\[
\Bar \uifrob_{n} := \Env(\Bar \uIFrob_n)
\]
be the (co)PROP envelope of the properadic bar construction $\Bar \uIFrob_n$. Note that this notation is slightly abusive, since the object is not equal to the bar construction of the PROP $\uifrob_{n}$. Then for $V$ a dg vector space consider the object 
\[
\Bar \uifrob_{n} \circ V.
\]
This comes with a natural "horizontal" commutative product and a compatible left $\Bar \uifrob_{n}$-coaction. This is called a monoid $\Bar \uifrob_{n}$-comodule in \cite{HLVold}.

Now, as shown in loc. cit., the datum of a $\uIFrob_{n,\infty}^+$-structure on $V$ is the same as a bidifferential on $\Bar \uifrob_{n} \circ V$, i.e., a degree 1 square zero linear map that is a derivation with respect to the monoid structure and a coderivation with respect to the $\Bar \uifrob_{n}$-coaction.
Indeed, since $\Bar \uifrob_{n} \circ V$ is cofree as a $\Bar \uifrob_{n}$-comodule cogenerated by $V^{\otimes \bullet}$ and free as a monoid generated by $\Bar \uIFrob_{n} \circ V$, any bidifferential $D$ is determined fully by its restriction to generators and projection to cogenerators 
\[
\alpha = \pi D\iota \in \Hom_S(\Bar \uifrob_{n} \circ V, V^{\otimes \bullet})\cong \Def^+(\uIFrob_{n\infty}, \End_V),
\]
and one can verify (see \cite{HLVold}) that the equation $D^2=0$ is equivalent to the Maurer-Cartan equation in the deformation complex.

\subsection{Removing the curvature}
\label{sec:removing curvature}
It is possible to remove the curvature in the deformation complexes above by twisting with a Maurer-Cartan element.
To this end, let us suppose that our vector space $V$ has the form $V=\Q 1 \oplus W$, i.e., we assume there is a distinguished element 1 and an associated augmentation.\footnote{In particular, note that the augmentation is canonical if $V^0$ is one-dimensional, which is notably the case if $V=H(M)$ is the cohomology of a connected space $M$.}
This data determines a canonical $\uIFrob_n$-structure on $V$ by declaring that $1$ is the unit element and that for all $x\in V$, $w,w'\in W$
\begin{align*}
1\cdot x &= x & w\cdot w' &=0 & \Delta x&= 0.
\end{align*}
This $\uIFrob_n$-structure determines a morphism of properads as the composition
\[
\phi\colon \uIFrob_{n,\infty}^+
\to \uIFrob_{n,\infty} \to \uIFrob_n\to \End(V),
\]
and hence a Maurer-Cartan element denoted by the same letter
\[
\phi \in \Def^+(\uIFrob_{n,\infty}, \End(V)).
\]
Twisting with this Maurer-Cartan element we obtain an honest (non-curved) dg Lie algebra 
\[
\Def^+(\uIFrob_{n,\infty}, \End(V))^\phi.
\]
Applying \eqref{equ:def map} to $\phi$ we then obtain a Maurer-Cartan element $\psi\in \Def(\Bar^c\GX_n, \End(V))$. By twisting \eqref{equ:def map} we then obtain a morphism of dg Lie algebras 
\begin{equation}\label{equ:def map 2}
\Def^+(\uIFrob_{n,\infty}, \End(V))^\phi\to \Def^+(\Bar^c\GX_n, \End(V))^\psi.
\end{equation}
\begin{prop}
    The morphism of dg Lie algebras \eqref{equ:def map 2} is a quasi-isomorphism.
\end{prop}
\begin{proof}
    This is a direct consequence of Proposition \ref{prop:Theta qiso}.
\end{proof}

\subsection{$\infty$-morphisms of $\uIFrob_{n,\infty}$-algebras and $\Bar^c\GX_n$-algebras}

Let $V$ and $W$ be graded vector spaces and let $\alpha$ (resp. $\beta$) be a $\uIFrob_{n,\infty}^+$-algebra structure on $V$ (resp. on $W$). We will encode these structures either as a Maurer-Cartan element in the deformation complex ($\alpha \in \Def^+(\uIFrob_{n,\infty}, \End_V)$) or equivalently as the corresponding bidifferential $D_\alpha$ on the monoid $\Bar \uifrob_{n}$-comodule $\Bar \uifrob_{n} \circ V$ (resp. $D_\beta$ on $\Bar \uifrob_{n} \circ W$).

Then an $\infty$-morphism $\Phi: V\to W$ is a morphism of monoid $\Bar \uifrob_{n}$-comodules
\[
\Phi \colon \Bar \uifrob_{n} \circ V \to \Bar \uifrob_{n} \circ W
\]
compatible with the differentials, i.e., 
\begin{equation}\label{equ:infty morphism}
\Phi D_\alpha= D_\beta \Phi.
\end{equation}
By freeness (as monoids) and cofreeness (as $\Bar \uifrob_{n}$-comodules) the data of an $\infty$-morphism $\Phi$ is fully encoded by its restriction to generators and projection to cogenerators 
\[
\phi :=\pi\Phi\iota \in 
\Hom_S(\Bar\uIFrob_n\circ V, W^{\otimes \bullet}) =: \fh_{(V,\alpha), (W,\beta)}.
\]
Furthermore, as has been shown in \cite{HLVold}, the right-hand side $\fh_{(V,\alpha), (W,\beta)}$ has a natural curved $L_\infty$-structure depending on $\alpha$ and $\beta$ such that the compatibility equation \eqref{equ:infty morphism} for $\Phi$ is equivalent to the Maurer-Cartan equation for $\phi$.
Hence the set of $\infty$-morphisms $V\to W$ is identified with the Maurer-Cartan set $\MC(\fh_{(V,\alpha), (W,\beta)})$.

Also note that from the description above $\infty$-morphisms can be naturally composed -- one just defines the composition as the composition of monoid $\Bar \uifrob_{n}$-comodule morphisms.

Finally, the same discussion applies mutatis mutandis to $\GX_n$ instead of $B\uIFrob_n$. 
We will denote the resulting curved $L_\infty$-algebra by $\tilde \fh_{(V,\alpha), (W,\beta)} =\Hom_S(\GX_n\circ V, W^{\otimes \bullet})$.

\subsection{The Hoffbeck-Leray-Vallette simplicial category $\HLVCat$}\label{sec:HLV simpl cat}
We next recall (and slightly rephrase) the construction of a simplicial category of $\uIFrob_{n,\infty}$-algebras due to Hoffbeck-Leray-Vallette \cite{HoffbeckLerayVallette}, which we call $\HLVCat$.
The objects of $\HLVCat$ are $\uIFrob_{n,\infty}$-algebras, that is, pairs $(V,\alpha)$ consisting of a graded vector space $V$ and a $\uIFrob^+_{n,\infty}$-structure $\alpha$ on $V$, for example encoded as a bidifferential $D_\alpha$ on $\Bar\uIFrob_n\circ V$.

For $(V,\alpha)$ and $(W,\beta)$ two objects of $\HLVCat$ the simplicial mapping space is defined as
\[
\Map_\bullet((V,\alpha),(W,\beta))
:=
\Mor_{\text{mon }\Bar\uIFrob_n\text{-comod}}((\Bar\uIFrob_n\circ V,D_\alpha),( (\Bar\uIFrob_n\circ V )\otimes \Omega(\Delta^\bullet),D_\beta) ),
\]
where the morphisms on the right are the monoid $\Bar\uIFrob_n$-comodule morphisms. Using the encoding of $\infty$-morphisms as MC elements of the previous subsection we can write the mapping spaces equivalently as Maurer-Cartan spaces 
\begin{align*}
\Map_\bullet((V,\alpha),(W,\beta))
&=
\MC( \Hom_S(\Bar\uIFrob_n\circ V, W^{\otimes *} \otimes \Omega(\Delta^\bullet) ) )
\\&\cong 
\MC( \Hom_S(\Bar\uIFrob_n\circ V, W^{\otimes *} )\hotimes \Omega(\Delta^\bullet) )
\\&= \MC_\bullet(\fh_{(V,\alpha), (W,\beta)}).
\end{align*}
To define the composition of the above simplicial mapping spaces it is most convenient to realize that $\Map_\bullet((V,\alpha),(W,\beta))$ is the same as the set of $\infty$-morphisms $V\otimes \Omega(\Delta^\bullet)\to W\otimes \Omega(\Delta^\bullet)$, after extending the ground ring $\Q$ to the differential graded ring $\Omega(\Delta^\bullet)$. Then the composition of the simplicial mapping spaces is just given by the obvious composition of $\infty$-morphisms, just defined over the ground ring $\Omega(\Delta^\bullet)$.

Again, we may replace $\Bar\uIFrob_n$ by $\GX_n$ in the above construction to obtain a simplicial category of $\Bar^c\GX_n$-algebras, i.e., pairs $(V,\alpha)$ consisting of a graded vector space $V$ and a Maurer-Cartan element $\alpha\in \Def^+(\Bar^c\GX_n, \End(V))$. The morphism spaces are defined as 
\[
\Map((V,\alpha), (W,\beta)) = \MC_\bullet(\tilde \fh_{ (V,\alpha), (W,\beta) }),
\]
with the composition defined analogously to above.
The morphism $\GX_n \to \Bar\uIFrob_n$ then induces a morphism of simplicial categories $\HLVCat \to \GXCat$, and from Proposition \ref{prop:Theta qiso} above we obtain:
\begin{cor}
The simplicial functor $\HLVCat \to \GXCat$ is an equivalence of $\infty$.categories.
\end{cor}

\section{Comparison of $\HLVCat$ and $\CAlg(\uifrob_n\Mod)$ and proof of Proposition \ref{prop:HLV equiv intro}}
\label{sec:HLV comparison}
\subsection{More on the model category $\uifrob_{n}\Mod$}
Let us call the model structure on $\uifrob_{n}\Mod$ of Lemma \ref{lem:uifrob model} the \emph{mixed} model structure. 
In this section we will also consider the projective model structure on the category $\uifrob_n\Mod$, which again exists by \cite[Theorem 3.2]{Keller}.
It is automatically a symmetric monoidal model structure.
Furthermore, the identity functors constitute a symmetric monoidal Quillen equivalence 
\[
\mathit{id}\colon \uifrob_n\Mod^{\text{mixed}} \leftrightarrows \uifrob_n\Mod^{\text{proj}}
\colon \mathit{id}
\]
between the two model structure, so we may work with either for the purposes of this section.
We also obtain a Quillen equivalence on the categories of commutative algebras
\[
\mathit{id}\colon \CAlg(\uifrob_n\Mod^{\text{mixed}}) \leftrightarrows \CAlg(\uifrob_n\Mod^{\text{proj}})
\colon \mathit{id}
\]


In the remainder of this section we shall use the projective structure and drop the superscript "proj".
The category $\uifrob_n\Mod$ is naturally a dg enriched model category, and by applying the Dold-Kan correspondence we obtain a simplicially enriched model category.
However, since the Dold-Kan correspondence is not symmetric monoidal (respectively, only weakly so), we prefer to work with a slightly different construction of the simplicial enrichment, following for example \cite{FWAut}. 
The price to pay is that we will not recover the full datum of a simplicial model category, but only an "almost" simplicial model category.

We endow all of the the categories above with a simplicial enrichment such that the mapping spaces satisfy
\[
\Map(X,Y) = \Mor_{\Omega(\Delta^\bullet)}
(X,Y),
\]
where the morphism sets on the right-hand side are computed in the category of (commutative algebra objects in) $\uifrob_n$-modules after extending the ground ring to $\Omega(\Delta^\bullet)$. We next note that $\uifrob_n\Mod$ is naturally cotensored over $\sSet$, with cotensoring defined as
\[
X^K = X\otimes \Omega(K),
\]
for $X\in \uifrob_n\Mod$ and $K\in \sSet$, with $\Omega(K)=\Mor_{\sSet}(K, \Omega(\Delta^\bullet))$ the differential forms on $K$.
We then have the adjunction formula 
\[
\Mor_{\sSet}
(K, \Map_{\uifrob_n\Mod}(X,Y)) = \Mor_{\uifrob_n\Mod}(X, Y^K).
\]
In particular, we have that 
\[
\Map_{\uifrob_n\Mod}(X,Y) = \Mor_{\uifrob_n\Mod}(X, Y^{\Delta^\bullet}).
\]
The functor $(-)^K$ preserves finite limits, but not arbitrary limits, and can thus not be a right adjoint. 
However, it is obviously lax symmetric monoidal and descends to the category $\CAlg(\uifrob_n\Mod)$, where it satisfies the same adjunction properties.

\begin{prop}
The bifunctor $(K,X)\mapsto X^K$ above satisfies the pullback-corner axiom. That is, given a cofibration $i\colon K \to L$ of simplicial sets and a fibration $p\colon X \to Y$ in $\uifrob_n\Mod$ or $\CAlg(\uifrob_n\Mod)$ the induced map
\[
X^L \to X^K \times_{Y^K} Y^L
\]
is a fibration, and a weak equivalence if either $i$ or $p$ is.
\end{prop}
\begin{proof}
    First note that weak equivalences, fibrations and all limits in $\CAlg(\uifrob_n\Mod)$ and $\uifrob_n\Mod$ are created object-wise in $\dgVect$. Thus it is sufficient to check the pullback-corner axiom for the bifunctor $(K,V)\mapsto V\otimes \Omega(K)$ on $\sSet\times\dgVect$. But there the axiom is well-known to hold.
\end{proof}

As a consequence, the object $X^{\Delta^\bullet}$ is a simplicial frame for $X\in \CAlg(\uifrob_n\Mod)$. Hence our ad hoc definition of the simplicial enrichment is compatible with the enrichment defined by the model structure in the sense that both mapping spaces are weakly equivalent.

\subsection{The functor from the Hoffbeck-Leray-Vallette category}
We next define a simplicial functor between simplicial categories
\[
F : \HLVCat \to \CAlg(\uifrob_n\Mod).
\]
Note that here we (ab)use the notation $\CAlg(\uifrob_n\Mod)$ for the simplicial category as in the preceding subsection.
To an object $(V,\alpha)$ of $\HLVCat$ we assign the object (bar resolution)
$$
F(V,\alpha) 
= ( \uifrob_n \circ (\Bar \uifrob_n) \circ V, d_{Bar} + d_\alpha),
$$
where the differential $d_{Bar}$ is the one from the bar construction and the Koszul part between $\uifrob_n$ and $\Bar \uifrob_n$ and $d_\alpha$ is defined by twisting with the Maurer-Cartan element $\alpha$.
Similarly, the functor $F$ sends a morphism $f\in \MC(\fh_{(V,\alpha),(W,\beta)}\hotimes \Omega(\Delta^\bullet))$ to the corresponding morphism $F(V,\alpha)\otimes \Omega(\Delta^\bullet)\to F(W,\beta)\otimes \Omega(\Delta^\bullet)$.

\subsection{Homotopical fully faithfulness of $F$}
\begin{lemma}
The objects in the image of $F$ are fibrant and cofibrant in the projective model structure.
\end{lemma}
\begin{proof}

They are fibrant because any object is fibrant. To check cofibrancy, consider the free/forgetful Quillen adjunction
\[
\Free \colon  \SSeq \leftrightarrows \CAlg(\uifrob_n\Mod) \colon U,
\]
where $\SSeq$ is the category of symmetric sequences.
Note that each object $F(V,\alpha)$ in the image of our functor $F$ is quasi-free, i.e.,
\[
F(V,\alpha) = ( Free( (\Bar \uifrob_n) \circ V), d ),
\]
and the generators are precisely $W:=\Bar \uIFrob_n \circ V$.
We endow the generators with the exhaustive ascending filtration such that $\mF^p W$ consists of elements with at most $p$ factors of $\uIFrob_n$ ("number of vertices" in the graphical picture) in the factor $\Bar \uifrob_n$.
The differential $d$ on $F(V,\alpha)$ then satisfies 
\[
d\mF^pW \subset \Free(\mF^{p-1} W) + \mF^pW,
\]
and from this cofibrancy of $F(V,\alpha)$ follows by a standard argument. To spell this out, write $F(V,\alpha) = \colim X_p$ with $X_p = ( Free(\mF^p W), d )$ and denote $Y_p = \mF^pW/\mF^{p-1}W$. Then $X_p$ fits into a pushout diagram of the form
\[
\begin{tikzcd}
( Free(Y_p[-1]), d ) \ar{r} \ar[hookrightarrow]{d} & X_{p-1} \ar{d} \\
( Free(Y_p\oplus Y_p[-1]), d ) \ar{r} & X_p
\end{tikzcd}.
\]
Since the left-hand vertical arrow is a cofibration so is the map $X_{p-1}\to X_p$ and hence $*\to \colim X_p=F(V,\alpha)$ is a transfinite composition of cofibrations and hence a cofibration.
\end{proof}

\begin{prop}
On simplicial mapping spaces the functor $F$ induces weak equivalences of simplicial sets.
\end{prop}
\begin{proof}
We compute the simplicial mapping space $\Map(F(V,\alpha),F(W,\beta))$ in $\CAlg(\uifrob_n\Mod)$.
Any morphism from a free object is determined by its restriction to generators, which are $\Bar \uIFrob_n \circ V$ by the proof of the previous lemma. 
Furthermore, one obtains by an essentially standard argument that 
\[
\Map(F(V,\alpha),F(W,\beta)) \simeq \MC_\bullet(\Hom(\Bar \uifrob_n \circ V, \uifrob_n\circ \Bar \uifrob_n \circ W)).
\]  
Here the dg Lie algebra structure is defined just as for the HLV dg Lie algebra $\fh_{(V,\alpha),(W,\beta)}$, but with $\uifrob_n\circ \Bar \uifrob_n \circ W$ on the output side instead of $W^{\otimes \bullet}$ as before. 
But both dg Lie algebras are obviously quasi-isomorphic, and the quasi-isomorphism is realized on the associated graded of the complete filtration by the total number of edges and vertices, and hence the Maurer-Cartan spaces are weakly equivalent by the Goldmann-Millson Theorem.
\end{proof}

\subsection{Homotopy essential image of $F$}
\begin{prop}\label{prop:ess surjectivity}
Let $X\in \CAlg(\uifrob_n\Mod)$ satisfy the Segal condition, and denote $H:=H(X(1))$. Then there exists a $\uIFrob_{n,\infty}$-algebra structure $\alpha$ on $H$ such that $X$ is weakly equivalent to $F(H,\alpha)$.
\end{prop}

Before giving the proof, it will be helpful to extend our construction of the model category $\CAlg(\uifrob_n\Mod)$ to $\uIFrob_{n,\infty}$ instead of $\uIFrob_n$.
To this end, consider the induction/restriction adjunction along the canonical quasi-isomorphism $f:\uifrob_{n,\infty}\to \uifrob_n$,
    \[
    f_* \colon \uifrob_{n,\infty}\Mod \leftrightarrows \uifrob_{n,}\Mod \colon f^*.
    \]
    We endow both categories with the projective model structures, which exist by \cite{Keller}. It is clear that $f^*$ creates weak equivalences (quasi-isomorphisms) and fibrations, so the adjunction is Quillen. By \cite[Proposition 3.2]{Toen} it is also a Quillen equivalence. Finally, induction is strong monoidal with respect to the Day monoidal structure and hence the right-adjoint is lax monoidal.
    We hence get an adjunction between the commutative algebra objects.
    \[
    f_* \colon \CAlg(\uifrob_{n,\infty}\Mod) \leftrightarrows \CAlg(\uifrob_{n,}\Mod) \colon f^*.
    \]
    We endow both sides with the model structure right transfered along the forgetful functors. Again, the right adjoint creates weak equivalences and fibrations and the adjunction is hence Quillen. By \cite[Theorem 4.19 and Corollary 3.6]{White} the adjunction is also a Quillen equivalence.

Furthermore, our functor $F:\HLVCat\to \CAlg(\uifrob_{n}\Mod)$ has an obvious extension 
\begin{gather*}
\tilde F\colon \to \CAlg(\uifrob_{n,\infty}\Mod)
\\
\tilde F(H,\alpha)  
= ( \uifrob_{n,\infty} \circ (\Bar \uifrob_n) \circ V, d_{\uifrob_{n,\infty}}+d_{Bar} + d_\alpha).
\end{gather*}
It satisfies $f_* \tilde F = F$, and furthermore, the map $\uifrob_{n,\infty}\to \uifrob_{n}$ induces a weak equivalence $\tilde F \Rightarrow f^* F$.

\begin{proof}[Proof of Proposition \ref{prop:ess surjectivity}]
    Let $X\in \CAlg(\uifrob_n\Mod)$ satisfy the Segal condition and let $H=H(X(1))$. Then by Theorem \ref{thm:htpy transfer} there is a $\uIFrob_{n,\infty}$-structure $\alpha$ on $H$ such that $H^{\otimes \bullet}$
    is weakly equivalent to $f^*M=M$ in $\CAlg(\uifrob_{n,\infty}\Mod)$.
    Next, we have a weak equivalence 
    \[
    \tilde F(H,\alpha) \to H^{\otimes \bullet}
    \]
    which is given on generators by the natural projection $\Bar \uIFrob_n\circ H\to H^{\otimes \bullet}$.
    By the Quillen equivalence above we hence conclude that (after applying the derived functor $f_*^L$ of $f_*$)
    \[
    f_*^L \tilde F(H,\alpha) \simeq f_* \tilde F(H,\alpha) = F(H,\alpha)
    \]
    is weakly equivalent to 
    \[
    f_*^L f^* M \simeq M
    \]
    in $\CAlg(\uifrob_n\Mod)$, thus showing the proposition.
\end{proof}

\section{Strongly unital homotopy Frobenius algebras, comparison to Abramyan's category, and proof of Theorem \ref{thm:main 2}}
\label{sec:strongly unital}

\subsection{Strongly unital $\uIFrob_{n,\infty}$-algebras and $\Bar^c\GX_n$-algebras}
The unit of a $\uIFrob_{n,\infty}$-algebra is a weak unit in general, in that the equation $1\cdot x = x$ only holds up to homotopy. It is however possible to restrict to a stronger notion of unit, without ``losing information'' in a homotopical sense.
We again assume that $V=\Q 1\oplus W$, where the element $1$ will be the unit of our $\uIFrob_{n,\infty}$-algebra. The splitting of $V$ also fixes an augmentation that makes it easier to define the relevant objects. However, at the end the choice of augmentation will be irrelevant,

Let us next consider the dg Lie algebra $\Def^+(\uIFrob_{n,\infty}, \End(V))^\phi$ as above, see Section \ref{sec:removing curvature}. This dg Lie algebra contains a dg Lie subalgebra 
\[
\fg_W\xhookrightarrow{\iota} \Def^+(\uIFrob_{n,\infty}, \End(V))^\phi
\]
defined as follows. Elements of $\fg_W$ are linear combinations (series) of graphs $\Gamma$ that satisfy:
\begin{enumerate}
    \item There are no univalent vertices in $\Gamma$.
    \item The input legs are all decorated by $W^*\subset V^*$.
\end{enumerate}
Such a graph $\Gamma\in \fg_W$ is considered an element $\iota(\Gamma)\in \Def^+(\uIFrob_{n,\infty}, \End(V))^\phi$ such that
\begin{equation}\label{equ:iota def}
\iota(\Gamma) = 
\sum_r
\frac{1}{r!}
\sum
\,
\begin{tikzpicture}
    \node[ext, inner sep=1mm] (v) at (0,0) {$\Gamma$};
    \node (r1) at (1,1) {$1^*$};
    \node (r2) at (1,0) {$\vdots$};
    \node (r3) at (1,-1) {$1^*$};
    \draw [->] (r1) edge (v) (r2) edge (v) (r3) edge (v);
    \draw [decorate,decoration={brace,amplitude=5pt,mirror,raise=2ex}]
  (1,-1.1) -- (1,1.1) node[midway,xshift=8mm]{$r \times$};
\end{tikzpicture}
\end{equation}
is obtained by attaching an arbitrary amount of input legs decorated by $1^*\in V^*$ to vertices of $\Gamma$.
Note that by the assumption that $\Gamma$ does not have univalent vertices, this attachment procedure does not create passing vertices.
One checks that $\fg_V\subset\Def^+(\uIFrob_{n,\infty}, \End(V))^\phi$ is a dg Lie subalgebra.
The differential on $\fg_V$ is combinatorially the same as that on $\Def^+(\uIFrob_{n,\infty}, \End(V))^\phi$ except that one drops all terms that produce univalent vertices or $1^*$-decorated input legs.

Note that any MC element in $\fg_V$ gives rise to an MC element in $\Def^+(\uIFrob_{n,\infty}, \End(V))$ via the inclusion $\iota$, and hence to a $\uIFrob_{n,\infty}^+$-structure on $V$. 
Note also that the definitions of $\phi$, of $\fg_W$ and of the inclusion $\iota$ use the choice of augmentation on $V$ that we made. However, the subspace $\iota(\fg_W)\subset \Def^+(\uIFrob_{n,\infty}, \End(V))^\phi$ is independent of this choice, as well as the subset of Maurer-Cartan elements that are in the image of $\iota$. In this way the following definition makes sense:

\begin{defi}\label{def:su}
    Let $V$ be a graded vector space and let $\alpha\in \MC(\Def^+(\uIFrob_{n,\infty}, \End(V)) )$ be a $\uIFrob_{n,\infty}$-algebra structure on $V$. Then we say that $\alpha$ is \emph{strongly unital} if it is in the image of $\iota$.
\end{defi}

\begin{prop}\label{prop:su equivalence}
Any $\uIFrob_{n,\infty}^+$-algebra $(V,\alpha)$ is $\infty$-quasi-isomorphic to a strongly unital one.
\end{prop}
\begin{proof}
    Let $(V,\alpha)$ be some $\uIFrob_{n,\infty}$-algebra, with $V$ a graded vector space and $\alpha\in \MC(\Def^+(\uIFrob_{n,\infty}, \End(V)) )$. We denote by $1\in V$ the unit in this structure and by 
    $d_V$ the differential, which we interpret as an element of the deformation complex.
    We fix a complement $W$ of $\Q1$ in $V$ so that $V=\Q1 \oplus W$. Note that the differential might possibly not preserve $W$.
    We then have $\alpha = \phi + d_V + \alpha'$, where $\alpha'$ is a linear combination of graphs with at least one vertex of valence $>1$.
    Note that $d_V1=0$ and hence $d_V$ is in $\fg_W$.
    We may then twist by $d_V$ and obtain the inclusion of curved dg Lie algebras 
    \[
    \iota \colon \fg_W^{d_V} \to \Def^+(\uIFrob_{n,\infty}, \End(V)) )^{\phi+d_V}.
    \] 
    Let us equip the curved dg Lie algebras on both sides with the descending complete filtrations $\mF^\bullet$ by the number of vertices of valence $>1$. By restriction we obtain the inclusion of curved dg Lie algebras
    \[
    \iota \colon \mF^1 \fg_W^{d_V} \to \mF^1\Def^+(\uIFrob_{n,\infty}, \End(V)) )^{\phi+d_V},
    \] 
    and $\alpha'$ is a Maurer-Cartan element in the right-hand curved dg Lie algebra.
    We claim that the associated graded of $\iota$ with respect to our filtration is a quasi-isomorphism. Let us suppose this is true for the moment. Then we may use the Goldmann-Millson Theorem for curved $L_\infty$-algebras \cite[Theorem 2.12]{RocaiLucio}, which in particular implies that inclusion $\iota$ induces a surjective map on the gauge equivalence classes of Maurer-Cartan elements. Hence there is a Maurer-Cartan element $\beta\in \mF^1 \fg_W^{d_V}$ such that $\iota(\beta)$ is gauge equivalent to $\alpha'$. Hence the Maurer-Cartan element $\iota(\beta+d_V)+\phi$, which by definition is a strongly unital $\uIFrob_{n,\infty}$-structure, is gauge equivalent to our given $\alpha$ in $\mF^1\Def^+(\uIFrob_{n,\infty}, \End(V)) )$.
    But gauge-equivalent MC elements encode $\infty$-quasi-isomorphic $\uIFrob_{n,\infty}$-algebra structures, so we are done.
    
    It remains to show the claim that the associated graded of $\iota$ with respect to our filtration is a quasi-isomorphism.
    Note that the associated graded differential on the left-hand side $\fg_W$ is just the part $d_V$, i.e., the differential on
    $V$. The differential on the associated graded of the deformation complex is $d_V$ plus a contribution $d'$ from the Lie bracket with the graph encoding the unit, which turns one $1^*$-decorated input leg into a univalent vertex.
    \begin{equation}\label{equ:d prime pic}
    d' \colon  
    \begin{tikzpicture}
    \node[int] (v) at (0,0) {};
    \node (r1) at (1,1) {$1^*$};
    \node (r2) at (1,0) {$\cdots$};
    \node (l1) at (-1,0) {$\cdots$};
    \draw[<-] (v) edge (r1) edge (r2);
    \draw[->](v) edge (l1); 
    \end{tikzpicture}
    \mapsto
    \begin{tikzpicture}
    \node[int] (v) at (0,0) {};
    \node (r2) at (1,0) {$\cdots$};
    \node (l1) at (-1,0) {$\cdots$};
    \node[int] (u) at (.7,.7) {};
    \draw[<-] (v) edge (r2);
    \draw[->](v) edge (l1) (u) edge (v); 
    \end{tikzpicture}
\end{equation}
    Taking a further spectral sequence by the total degree of all decorations\footnote{Our degrees may be unbounded, but since the complex splits into a direct product of finite dimensional complexes the spectral sequence still converges to the cohomology.}, the associated graded differential is just $d'$.
    Now $d'$ has an obvious homotopy that makes a univalent vertex into a $1^*$-decorated input leg. Hence the cohomology is concentrated in the subspace where neither a univalent vertex nor a $1^*$-decorated input leg is present, which is precisely $\fg_W$. Hence $\iota$ induces a quasi-isomorphism on the associated graded, as claimed.
\end{proof}

\begin{rem}
Note that the same proof also shows that for any MC element $\beta\in \fg_W$ the inclusion $\iota\colon \fg_W^\beta \to \Def^+(\uIFrob_{n,\infty}, \End(V))^{\phi+\iota(\beta)}$ is a quasi-isomorphism.
In particular, $\iota$ induces a weak homotopy equivalence on each connected component of the Maurer-Cartan spaces. (This fact will not be used in this paper.)
\end{rem}




The above construction applies mutatis mutandis to $\Bar^c \GX_n$ instead of $\uIFrob_{n,\infty}$. Concretely, we define a dg Lie subalgebra 
\[
\tilde \fg_W\xhookrightarrow{\iota} \Def(\Bar^c \GX_n, \End(V))^\psi,
\]
whose elements are linear combinations of graphs without univalent vertices, and with all inpute legs decorated by $W^*\subset V^*$.
The inclusion $\iota$ is again  defined by \eqref{equ:iota def}, just with undirected edges.
Our morphism of dg Lie algebra \eqref{equ:def map 2} then restricts to a morphism of dg Lie algebras
\begin{equation}\label{equ:def map 3}
    \fg_W\to \tilde\fg_W.
\end{equation}
Proposition \ref{prop:Theta qiso} then readily implies
\begin{cor}
    The morphism of dg Lie algebras \eqref{equ:def map 3} is a quasi-isomorphism.
\end{cor}

\subsection{$\infty$-morphisms of strongly unital $\uIFrob_{n,\infty}$-algebras and $\Bar^c\GX_n$-algebras}
Let $(V,\alpha)$ and $(W,\beta)$ be two strongly unital $\uIFrob_{n,\infty}$-algebras. We want to define the notion of strongly unital $\infty$-morphisms between these objects.
To this end, we again fix a splitting $V=\Q 1 \oplus V'$, $W=\Q 1 \oplus W'$, with the implicit understanding that our eventual definition will be independent of the choice of splittings.

Now consider the subspace 
\[
\fh_{(V,\alpha), (W,\beta)}^{su}\subset \fh_{(V,\alpha), (W,\beta)}
\]
consisting of series of graphs $\Gamma$ such that 
\begin{enumerate}
    \item There are no univalent vertices in $\Gamma$.
    \item The (say $k$) input legs are decorated by $((V')^{\otimes k})^*\subset (V^{\otimes k})^*$.
\end{enumerate}
The output legs may carry arbitrary decorations in $W$.
Again we define a non-trivial inclusion 
\[
\iota:\fh_{(V,\alpha), (W,\beta)}^{su}\hookrightarrow \fh_{(V,\alpha), (W,\beta)}
\]
by summing over all ways of attaching $1^*$-decorated input legs.
One checks that the image 
\begin{equation}\label{equ:iota morphism}
\iota(\fh_{(V,\alpha), (W,\beta)}^{su})
\subset 
\fh_{(V,\alpha), (W,\beta)}
\end{equation}
is well-defined (i.e., independent of the choice of splitting) and a curved $L_\infty$-subalgebra.

\begin{defi}
    Let $(V,\alpha)$ and $(W,\beta)$ be two strongly unital 
    $\uIFrob_{n,\infty}$-algebras, 
    and let $\phi\in \MC(\fh_{(V,\alpha), (W,\beta)})$ be an $\infty$-morphism.
    Then we say that $\phi$ is \emph{strongly unital} if $\phi$ is in the image of $\iota$.
\end{defi}

\begin{prop}\label{prop:fh inclusion}
    The inclusion \eqref{equ:iota morphism} induces a weak equivalence on the associated graded complexes with respect to the complete descending filtrations by the number of edges plus non-univalent vertices. This remains true if we twist $\fh_{(V,\alpha), (W,\beta)}^{su}$ by a Maurer-Cartan element, and we twist $\fh_{(V,\alpha), (W,\beta)}$ by the image of that Maurer-Cartan element under the inclusion.
\end{prop}
\begin{proof}
    Note that any term of $\alpha$ and $\beta$ that has at least one non-univalent vertex contributes a term to the (twisted) differential on the associated graded complexes that increases the number of non-univalent vertices. There remain only the terms of $\alpha$ and $\beta$ that encode the differentials on $V$ and $W$, which we denote by $d_V$ and $d_W$, and those that encode the units. 
    The contribution of the units in turn is a piece of the differential $d'$ that replaces one $1^*$-decorated input leg by a univalent vertex as in \eqref{equ:d prime pic}.
    The remainder of the proof is identical to that of Proposition \ref{prop:su equivalence}: Taking a further spectral sequence on the totatl decoration degrees we mask out the contribution of $d_V$ and $d_W$, and since $d'$ has a homotopy the cohomology is concentrated in the subspace where neither a univalent vertex nor a $1^*$-decorated input leg is present, which is precisely $\fh_{(V,\alpha), (W,\beta)}^{su}$.
\end{proof}

From Proposition \ref{prop:fh inclusion} and the Goldman-Millson Theorem we then immediately obtain:
\begin{cor}\label{cor:fh inclusion}
The inclusion \eqref{equ:iota morphism} induces a weak equivalence 
\[
\MC_\bullet(\fh_{(V,\alpha), (W,\beta)}^{su}) \simeq 
\MC_\bullet(\fh_{(V,\alpha), (W,\beta)}).
\]
\end{cor}

\begin{rem}
The choice of the filtration by the number of edges plus the number of non-univalent vertices is such that the whole curved $L_\infty$-algebra is concentrated in filtration degrees $\geq 1$, as is required to apply the Goldman-Millson Theorem.  Apart from this issue, filtering only by the number of non-univalent vertices would also work for the arguments above.
\end{rem}

\subsection{Definition of the categories $\HLVCat^{su}$ and $\GXCat^{su}$}
We may define the simplicial subcategory 
\[
\HLVCat^{su} \subset \HLVCat
\]
of strongly unital $\uIFrob_{n,\infty}$-algebras as follows:
The objects of $\HLVCat^{su}$ are those $\uIFrob_{n,\infty}$-algebras $(V,\alpha)$ that are strongly unital according to Definition \ref{def:su}. For $(V,\alpha)$ and $(W,\beta)$ strongly unital $\uIFrob_{n,\infty}$-algebras, we define the morphism simplicial sets in $\HLVCat^{su}$ to be 
\[
\Map_{su}((V,\alpha), (W,\beta)) := \MC_\bullet(\fh_{(V,\alpha),(W,\beta)}^{su}).
\]
We then immediately see that we have a simplicial inclusion $\HLVCat^{su} \subset \HLVCat$. By Proposition \ref{prop:su equivalence} the inclusion functor is essentially surjective, and by Corollary \ref{cor:fh inclusion} it is also homotopically fully faithful.
We hence immediately obtain:
\begin{cor}
    The inclusion $\HLVCat^{su} \subset \HLVCat$ is an equivalence of $\infty$-categories.
\end{cor}



Likewise, replacing $\Bar\uIFrob_n$ by $\GX_n$ we obtain a simplicial category $\GXCat^{su}$. It fits into a commutative square of simplicial categories 
\[
\begin{tikzcd}
\HLVCat^{su} \ar{r}\ar{d} & \GXCat^{su} \ar{d} \\
\HLVCat \ar{r} & \GXCat
\end{tikzcd}
\]
in which all functors are essentially surjective and homotopically fully faithful, i.e., equivalences of $\infty$-categories.



\subsection{Comparison to Abramyan's category $\GCA$}
With the above prerequisites we can now explain the relation of our categories to a simplicial category defined by Abramyan \cite{Abramyan}, which we denote $\GCA$.
Objects of $\GCA$ are pairs $(V,\alpha)$ consisting of a finite dimensional, non-negatively graded vector spaces $V$ and a Maurer-Cartan element $\alpha\in \tilde \fg_V$.
The simplicial mapping spaces are defined as 
\[
\Map_{\GCA}((V,\alpha), (W,\beta))
=
\MC_\bullet(\tilde \fh_{(V_1,\alpha), (W_1,\beta)}^{su}).
\]
Hence we see that there is fully faithful simplicial functor 
\[
\GCA \to \GXCat
\]
that realizes an equivalence of $\GCA$ with the subcategory of objects for which the underlying graded vector space is finite dimensional and non-negatively graded.



\section{Counital Frobenius algebras}
\label{sec:counital}
\subsection{Definitions}

We denote by $\ucIFrob_n$ the (unital and counital) involutive Frobenius properad, defined such that 
\[
\ucIFrob_n(r,s) = \Q[n(r-1)]\otimes \sgn_r^{\otimes n}
\]
for all $r,s\geq 0$. There is a natural properad map $\uIFrob_n\hookrightarrow \ucIFrob_n$, so that any $\ucIFrob_n$-algebra is also naturally a $\uIFrob_n$-algebra.
Conversely, one may ask which $\uIFrob_n$-algebra structures can be extended to $\ucIFrob_n$-algebra structures. The following result describes a list of equivalent criteria.
\begin{lemma}\label{lem:non-degenerate}
   Let $H$ be a $\uIFrob_n$-algebra with zero differential.
   Let $\Delta\in H^{\otimes 2}$ and $u\in H$ be the diagonal and unit of $H$. Then the following are equivalent:
   \begin{enumerate}
       \item There is a $\ucIFrob_n$-algebra structure on $H$ that restricts to the given $\uIFrob_n$-algebra structure.
       \item The diagonal is non-degenerate, i.e., it induces a bijection $H^*\to H$.
       \item There is a linear map $c:H\to \Q$ such that 
        \begin{equation}
        \label{equ:cohom counit}
        (\mathit{id}\otimes c)\Delta = u \in H.
        \end{equation}
   \end{enumerate}
\end{lemma}
Also note that a $\ucIFrob_n$-algebra must necessarily be finite dimensional.
\begin{proof}
    (1)$\Rightarrow$ (2): This is clear since the inverse map in the Frobenius algebra is given by the generating ("cap") operation in $\ucIFrob_n(0,2)$.

    (2)$\Rightarrow$ (3) is obvious, since (3) just asks for $u$ to be in the image of the morphism $H^*\to H$.

    (3) $\Rightarrow$ (1): We have to check that the linear map 
    \begin{gather*}
        H \to H^*
        \\
        h\mapsto (x\mapsto c(hx))
    \end{gather*}
    is a bijection. But one easily checks that the inverse map is the map $H^*\to H$ given by $\Delta$, using that \eqref{equ:cohom counit} holds.
\end{proof}

We will also need a version for morphisms.
\begin{lemma}\label{lem:ifrob morph}
Let $H$ and $H'$ be $\ucIFrob_n$-algebras with diagonals $\Delta, \Delta'$, pairings $b,b'$, units $u,u'$ and counits $c,c'$. Let $\phi:H\to H'$ be a linear map that is a morphism of $\ucIFrob_n$-algebras. Then the following are equivalent:
\begin{enumerate}
    \item $f$ is a morphism of $\ucIFrob_n$-algebras, i.e., $c'\circ f = c$.
    \item $f$ is an isomorphism.
    \item There is a linear map $\lambda:W\to \Q$ such that $(\mathit{id}\otimes \lambda f)\Delta = u$.
\end{enumerate}
\end{lemma}
\begin{proof}
    (1) $\Rightarrow$ (2): We have $(f\otimes f)\Delta=\Delta'$, hence by non-degeneracy of $\Delta'$ the rank of $f$ must be equal to the dimension of $W$. Dually $b' (f\otimes f)=b$, hence the rank of $f$ must be equal to the dimension of $V$ and $f$ is hence an isomorphism.
    
    (2) $\Rightarrow$ (3): This follows directly from the non-degeneracy of $\Delta$. 

    (3)$\Rightarrow$ (1): Applying $b$ to the equation we see that $\lambda f = c$. Hence we are done if we can show that $\lambda=c'$.
    On the other hand, applying $f$ to the equation we obtain $(\mathit{id} \otimes \lambda)\Delta' = (f \otimes \lambda f)\Delta = fu = u'$ and conclude that $\lambda = c'$ using non-degeneracy of $\Delta'$.
\end{proof}

Next, we consider homotopy $\ucIFrob_n$-algebras. The discussion of Section \ref{sec:cofib uIFrob} extends naturally to the case of $\ucIFrob_n$, and we may define a cofibrant resolution in the category of properads
\[
\ucIFrob_{n,\infty} := \Bar^c \Bar \ucIFrob_{n} 
\]
by the (curved) properadic bar-cobar construction. Graphically, elements of $\Bar \ucIFrob_{n}(r,s)$ may be considered as linear combinations of diagrams as in Section \ref{sec:GX}, with the only difference being that now these diagrams may have $r=0$ output legs and may contain vertices (including univalent) vertices with no output edges.

\begin{defi}\label{def:non-degenerate}
Let $(V,\alpha)$ a $\uIFrob_{n,\infty}$-algebra, and let $H:=H(V)$ be the underlying graded $\uIFrob_{n}$-algebra.
Then we call $(V,\alpha)$ \emph{non-degenerate} if $H$ satisfies either of the equivalent conditions of Lemma \ref{lem:non-degenerate}. Otherwise we call $(V,\alpha)$ \emph{degenerate}.
\end{defi}



\subsection{Another cofibrant replacement of $\ucIFrob_n$}
\label{sec:Xn}
Let $\FreeP(\Delta, u, c, h)$ be the free properad in the following four generators.
\begin{itemize}
    \item The diagonal $\Delta$ is an arity $(2,0)$-generator (i.e., no inputs, two outputs) of degree $n$ that is $(-1)^n$-symmetric.
    \item The unit $u$ is a degree zero generator of arity $(1,0)$.
    \item The counit $c$ is a degree $-n$-generator of arity $(0,1)$.
    \item The element $h$ is of degree $-1$ and arity $(1,0)$.
\end{itemize}
We define a differential $d$ on $\FreeP(\Delta, u, c, h)$ such that on generators 
\begin{align*}
    d \Delta &= du =dc =0 \\
    dh &= (\mathit{id}\otimes c) \circ \Delta - u.
\end{align*}
We consider the following diagram of properad morphisms. with $X_n$ being defined as the pushout of the upper-left corner:
\[
\begin{tikzcd}
\FreeP(\Delta, u) \ar{r} \ar[hookrightarrow]{d} & \uIFrob_n \ar{d} \ar{ddr} & \\ 
(\FreeP(\Delta, u, c, h) , d) \ar{r} \ar{drr} & X_n \ar[dashed]{dr}{f}& \\
& & \ucIFrob_n
\end{tikzcd}.
\]
The upper horizontal arrow sends $u$ to the unit and $\Delta$ to the generator of $\uIFrob_n$ (the "diagonal").
The left vertical arrow is the obvious inclusion. The lower diagonal arrow sends $h$ to 0, $u$ to the unit, $c$ to the counit and $\Delta$ to the $(2,0)$-generator in $\ucIFrob_n$. The dashed arrow $f$ arises from the universal property of the pushout.

\begin{lemma}\label{lem:Xn qiso}
    The arrow $f:X_n\to \ucIFrob_n$ is a quasi-isomorphism.
\end{lemma}
\begin{proof}
    Let us first describe the dg vector space $X_n(k,l)$. Elements are linear combinations of formal compositions of elements of $\uIFrob_n$ with the operations $c$ and $h$. 
    \[
    \begin{tikzpicture}
        ...
    \end{tikzpicture}
    \]
    Hence we have 
    \begin{align*}
    X_n(k,l) &= 
    \bigoplus_{i,j\geq 0 \atop k+i\geq 1} \uIFrob_n(k+i,l+j)[in+j]\otimes_{S_i\times S_j} \sgn_i^{\otimes n} \otimes \sgn_j
    \\&=
    \bigoplus_{i\geq 0 \atop k+i\geq 1} \left( \underbrace{
    \uIFrob_n(k+i,l)[in]\otimes_{S_i} \sgn_i^{\otimes n}
    }_{=\Q a_{k,l,i}}
    \oplus 
    \underbrace{
    \uIFrob_n(k+i,l+1)[in+1]\otimes_{S_i} \sgn_i^{\otimes n}
    }_{=\Q b_{k,l,i}}
    \right)
    .
    \end{align*}
    Note that each summand in the last line is one-dimensional, and we denote the generators by $a_{k,l,i}$, $b_{k,l,i}$ as indicated. Concretely, $a_{k,l,i}$ is the composition of the generator of $\uIFrob_n(k+i,l)$ with $i$ copies of $u$ and has degree $(k-1)n$. Similarly $b_{k,l,i}$ is obtained by  composing the generator of $\uIFrob_n(k+i,l+1)$ with $i$ copies of $u$ and one copy of $h$ resulting in an element of degree $(k-1)n-1$.
    One easily checks that the differential takes on the form 
    \[
    d b_{k,l,i} = a_{k,l,i+1} - a_{k,l,i}.
    \]
    To compute the cohomology of this complex we endow $X_n(k,l)$ with the exhaustive ascending filtration obtained by declaring $a_{k,l,i}$ to be of filtration degree $i$ and $b_{k,l,i}$ of filtration degree $i+1$. Considering the associated spectral sequence, the differential on the associated graded is then 
    \[
    d b_{k,l,i} = a_{k,l,i+1}.
    \]
    Since this differential has an obvious homotopy, the cohomology of $X_n(k,l)$ is easily evaluated to the following:
    \begin{itemize}
        \item If $k>0$ then $\gr_m X_n(k,l)$ is acyclic for each $m>0$, and thus $H(X_n(k,l))\cong \Q a_{k,l,0}$ is one-dimensional represented by $a_{k,l,0}$, i.e., the generator of $\uIFrob_n(k,l)$.
        \item If $k=0$ then note that the generator $a_{0,l,0}$ and $b_{0,l,0}$ are not present. Hence $\gr_0X_n(0,l)=0$, while still $H(\gr_m X_n(k,l))=0$ for $m\geq 2$ by the same argument as above. But $\gr_1X_n(0,l)=\Q a_{0,l,1}$ is 1-dimensional, and hence $H(X_n(0,l))\cong \Q a_{0,l,1}$.
    \end{itemize}
    It is clear that the morphism $f$ sends each $a_{k,l,i}$ to the generator of $\ucIFrob_n(k,l)$, and hence $f$ is a quasi-isomorphism.
\end{proof}

Next we similarly define a second dg properad $X_{n,\infty}$ as the pushout in the following diagram
\begin{equation}\label{equ:Xni}
\begin{tikzcd}
\FreeP(\Delta, u) \ar{r} \ar[hookrightarrow]{d} & \uIFrob_{n,\infty} \ar{d} \ar{ddr} & \\ 
(\FreeP(\Delta, u, c, h) , d) \ar{r} \ar{drr} & X_{n,\infty} \ar[dashed]{dr}{f_\infty}& \\
& & \ucIFrob_{n,\infty}
\end{tikzcd}.
\end{equation}
The maps are defined analogously to before, in particular, the unit $u$ counit $c$ and diagonal $\Delta$ are sent to the respective generators of $\uIFrob_{n,\infty}$ or respectively $\ucIFrob_{n,\infty}$.
Just mind that the lower diagonal arrow maps the element $h$ as follows 
\[
h\mapsto \,
\begin{tikzpicture}
    \node[int] (v) at (-.5,.5) {};
    \node[int] (w) at (0,0) {};
    \node (o) at (-1.2,0) {$\scriptstyle 1$};
    \draw[->] (w) edge (v) edge (o);
\end{tikzpicture}
\]
so as to make the lower diagonal map compatible with the differentials. 

Note the following facts:
\begin{itemize}
    \item Since the left-hand vertical arrow in \eqref{equ:Xni} is a cofibration, so is the right-hand vertical arrow. Hence $X_{n,\infty}$ is cofibrant since $\uIFrob_{n,\infty}$ is.
    \item By functoriality of the pushout we have the following commutative square 
    \begin{equation}\label{equ:X square}
    \begin{tikzcd}
    X_{n,\infty} \ar{r} \ar{d} & X_n \ar{d}{\sim}  \\
    \ucIFrob_{n,\infty} \ar{r}{\sim} & \ucIFrob_n
    \end{tikzcd},
    \end{equation}
    with the right-hand vertical arrow a quasi-isomorphism by Lemma \ref{lem:Xn qiso} above.
\end{itemize}

\begin{lemma}\label{lem:Xni qiso}
All arrows of \eqref{equ:X square} are quasi-isomorphisms.
\end{lemma}
\begin{proof}
    Note that it suffices to check that the map $X_{n,\infty} \to X_n$ is a quasi-isomorphism.
    For the explicit description of $X_n$ we refer back to the proof of Lemma \ref{lem:Xn qiso}. Similarly, we see that 
\begin{align*}
    X_{n,\infty}(k,l) &= 
    \bigoplus_{i,j\geq 0 \atop k+i\geq 1} \uIFrob_{n,\infty}(k+i,l+j)[in+j]\otimes_{S_i\times S_j} \sgn_i^{\otimes n} \otimes \sgn_j
    \\&=
    \bigoplus_{i\geq 0 \atop k+i\geq 1} \left( 
    \uIFrob_n(k+i,l)[in]\otimes_{S_i} \sgn_i^{\otimes n}
    \oplus 
    \uIFrob_n(k+i,l+1)[in+1]\otimes_{S_i} \sgn_i^{\otimes n}
    \right)
    .
    \end{align*}
    The differential has the form $d_{\infty} + d$, where $d_{\infty}$ is the differential on $\uIFrob_{n,\infty}$ and $d$ is the part that removes an $h$-generator as in the proof of Lemma \ref{lem:Xn qiso}. Now consider the morphism 
    \[
    X_{n,\infty} \to X_{n}
    \]
    and equip both sides with the ascending complete filtration by the number of $h$-generators. The associated graded morphism is idnetified with
    \[
    (X_{n,\infty} , d_\infty) \to (X_n,0)
    \]
    and clearly a quasi-isomorphism by the Künneth formula. It hence follows that $X_{n,\infty}\to X_n$ is a quasi-isomorphism.
\end{proof}

\begin{cor}
    The diagram 
    \[
\begin{tikzcd}
    \FreeP(\Delta,u) \ar[d] \ar[r] &  \uIFrob_n \ar[d] \\
    (\FreeP(\Delta,u,c,h),d) \ar[r] & \ucIFrob_n  
\end{tikzcd}
\]
is a homotopy pushout square in the category of dg properads.
\end{cor}

Now note that the proofs of Lemmas \ref{lem:Xn qiso} and \ref{lem:Xni qiso} also go through if we had replaced the role of $\uIFrob_n$ by $\ucIFrob_n$ from the start. In this case the case distinction $k=0$ versus $k>0$ appearing in the proof of Lemma \ref{lem:Xn qiso} would not even have been necessary.
Hence we obtain:

\begin{cor}\label{cor:htpy epi}
    All three squares in the diagram 
    \[
\begin{tikzcd}
    \FreeP(\Delta,u) \ar[d] \ar[r] &  \uIFrob_n \ar[r] \ar[d] & \ucIFrob_n \ar[d] \\
    (\FreeP(\Delta,u,c,h),d) \ar[r] & \ucIFrob_n \ar[r] & \ucIFrob_n
\end{tikzcd}
\]
are homotopy pushout squares in the category of dg properads.
In particular, the inclusion $\uIFrob_n\to \ucIFrob_n$ is a homotopy epimorphism.\footnote{This just means that the right-hand square is a homotopy pushout square.}
\end{cor}

\subsection{Comparing $\uIFrob_{n,\infty}$- and $\ucIFrob_{n,\infty}$-algebras }
Let $V$ be a dg vector space, $\POp$ a properad and let $\alpha,\beta:\POp\to \End_V$ be properad morphisms, thus defining two $\POp$-algebra structures on $V$. We say that the $\POp$-algebras $(V,\alpha)$ and $(V,\beta)$ are isotopic if $\alpha$ and $\beta$ are in the same path component of the properadic mapping space 
\[
\Map(\POp,\End_V).
\]

Next, let $\POp=\uIFrob_{n,\infty}$ and $\alpha:\uIFrob_{n,\infty}\to \End_V$.
Then we would like to determine whether $(V,\alpha)$ can be extended to a $\ucIFrob_{n,\infty}$-algebra. More precisely, we have a morphism of properadic mapping spaces 
\[
\Map(\ucIFrob_{n,\infty},\End_V) \to \Map(\uIFrob_{n,\infty},\End_V),
\]
and an associated map on the connected components 
\begin{equation}\label{equ:pi0 map}
\pi_0\Map(\ucIFrob_{n,\infty},\End_V) \to \pi_0\Map(\uIFrob_{n,\infty},\End_V).
\end{equation}
and we would like to determine whether the component of $\alpha$ is in the image of the above morphism. In this case we say that $(V,\alpha)$ can be extended to a $\ucIFrob_{n,\infty}$-structure up to isotopy.
If so, then it is clear that $(V,\alpha)$ must be non-degenerate in the sense of Definition \ref{def:non-degenerate}. The converse also holds, as the following result shows.

\begin{prop}\label{prop:ucIF uIF fiber}
    Let $(V,\alpha)$ be a $\uIFrob_{n,\infty}$-algebra.
    Then the morphism of mapping spaces of properads 
    \begin{equation}\label{equ:map map}
    \Map(\ucIFrob_{n,\infty},\End_V) \to \Map(\uIFrob_{n,\infty},\End_V)
    \end{equation}
    is a fibration. The fiber over $\alpha$ is either empty or contractible. It is contractible iff $(V,\alpha)$ is non-degenerate in the sense of Definition \ref{def:non-degenerate}.
    In particular, $(V,\alpha)$ can be extended to a $\ucIFrob_{n,\infty}$-structure up to isotopy iff $(V,\alpha)$ is non-degenerate.
\end{prop}
\begin{proof}
    The map \eqref{equ:map map} is a fibration because it is induced by the cofibration $\uIFrob_{n,\infty}\to \ucIFrob_{n,\infty}$ and $\End_V$ is fibrant, as is any properad in the projective model structure on properads.
    Also note that for the same reason and since  $\ucIFrob_{n,\infty}$ is cofibrant the spaces $\Map(\ucIFrob_{n,\infty},\End_V)$ and $\Map(\uIFrob_{n,\infty},\End_V)$ are models for the derived mapping spaces $\Map^h(\ucIFrob_{n},\End_V)$ and $\Map^h(\uIFrob_{n},\End_V)$ in the category of dg properads.

    We next show that the fiber is either contractible or empty.
    In short, this follows from Corollary \ref{cor:htpy epi} and the fact that applying $\Map^h(-,\End_V)$ to a homotpy epimorphism produces a homotopy monomorphism.
    Spelled out, applying $\Map^h(-,\End_V)$ to the objects of the right-hand homotopy pushout square of Corollary \ref{cor:htpy epi} we obtain the homotopy pullback square
    \[
    \begin{tikzcd}
    \Map^h(\ucIFrob_n,\End_V) \ar{r} \ar{d} 
    & \Map^h(\ucIFrob_n,\End_V) \ar{d} \\
    \Map^h(\ucIFrob_n,\End_V)\ar{r} 
    &
    \Map^h(\uIFrob_n,\End_V)
    \end{tikzcd}.
    \]
    Hence the desired result follows.
    
    Finally, we have to check that the fiber is non-empty iff $(V,\alpha)$ is non-degenerate. To this end, consider first the diagram of properad morphisms
    \[
    \begin{tikzcd}
    X_{n,\infty} \ar{d}{\sim}  & \uIFrob_{n,\infty} \ar[hookrightarrow]{l} \ar[hookrightarrow]{ld}\\
    \ucIFrob_{n,\infty} & 
    \end{tikzcd}.
    \]
    The left-hand vertical arrow is a quasi-isomorphism by Lemma \ref{lem:Xni qiso}, all objects are cofibrant, and the two right-hand arrows are cofibrations. Hence applying $\Map(-,\End_V)$ to the above diagram we get the diagram of simplicial sets
    \[
    \begin{tikzcd}
    \Map(X_{n,\infty},\End_V) \ar[twoheadrightarrow]{r} & \Map(\uIFrob_{n,\infty},\End_V)  \\
    \Map(\ucIFrob_{n,\infty},\End_V) \ar{u}{\sim} \ar[twoheadrightarrow]{ur} & 
    \end{tikzcd}
    \]
    in which the left-hand arrow is still a weak equivalence of simplicial sets and the right-hand arrows are fibrations. 
    The fibers of both fibrations are hence weakly equivalent, and we my compute the fiber of the upper arrow instead.
    Now by applying $\Map(-,\End_V)$ to the pushout diagram \eqref{equ:Xni} defining $X_{n,\infty}$ we obtain a pullback diagram 
    \begin{equation}\label{equ:Xni 2}
    \begin{tikzcd}
    \Map(X_{n,\infty},\End_V) \ar[twoheadrightarrow]{r}\ar{d} 
    &
    \Map(\uIFrob_{n,\infty},\End_V)\ar{d}{\pi} \\
    \Map((\FreeP(\Delta, u, c, h) , d), \End_V) 
    \ar[twoheadrightarrow]{r}
    &
    \Map(\FreeP(\Delta, u),\End_V)
    \end{tikzcd}.
\end{equation}
In particular, given any point $\alpha\in \Map(\uIFrob_{n,\infty},\End_V)$, the fiber of the upper horizontal arrow over $\alpha$ is the same as the fiber over $\pi\alpha$ of the lower horizontal arrow.
We need to show that this fiber is non-empty iff $(V,\alpha)$ is non-degenerate.
To this end let $\Delta_V\in V\otimes V$ and $u_V\in V$ the diagonal and counit as given by $\pi\alpha$.
Then the fiber is nonempty if there is a closed element $c\in V^*$ of degree $-n$ and some $h_V\in V$ of degree -1 such that 
\[
dh_V = (\mathit{id}\otimes c)\Delta_V - u_V.
\]
But this is equivalent to the condition \eqref{equ:cohom counit} to hold on cohomology. Hence the fiber is indeed nonempty iff $(V,\alpha)$ is non-degenerate.

\end{proof}



\subsection{Mapping spaces}
We denote by $\cHLVCat$ the Hoffbeck-Leray-Vallette simplicial category of $\ucIFrob_{n,\infty}$-algebras. Concretely, the objects of $\cHLVCat$ are $\ucIFrob_{n,\infty}$-algebras, and the morphism simplicial sets $\Map((V,\alpha),(W,\beta))$ may be defined as the $\infty$-morphisms $V\otimes \Omega(\Delta^\bullet)\to W\otimes \Omega(\Delta^\bullet)$ after extending the ground ring to $\Omega(\Delta^\bullet)$, cf. the discussion in Section \ref{sec:HLV simpl cat}.
We have a natural forgetful (simplicial) functor. 
\begin{equation}\label{equ:cHLV functor}
\cHLVCat \to \HLVCat.
\end{equation}

Now consider two $\ucIFrob_{n,\infty}$-algebras $(V,\alpha)$ and $(W,\beta)$. We would like to compare the mapping space as $\ucIFrob_{n,\infty}$-algebras to the mapping space of the underlying $\uIFrob_{n,\infty}$-algebras. Before doing that, a few remarks are in order.
First, any $\infty$-morphism of $\ucIFrob_{n,\infty}$-algebras induces on cohomology a plain morphism of $\ucIFrob_n$-algebras. It is elementary that any such morphism is necessarily an isomorphism, see also Lemma \ref{lem:ifrob morph}.
Hence all morphisms in $\cHLVCat$ are in fact quasi-isomorphisms. This is not true in $\HLVCat$, even between objects that are $\ucIFrob_{n,\infty}$-algebras.

\begin{ex}
To give a counterexample, consider the "trivial" $\ucIFrob_0$-algebra $A=\Q$. Let $B\cong A\oplus A\cong \Q\oplus \Q$ the product of two copies of $A$ in the category of $\uIFrob_0$-algebras. The object $B$ is in fact a $\ucIFrob_0$-algebra. (Take as the counit the sum of the two counits of the constituents.)
We also have the two natural maps $B\to A$ of $\uIFrob_0$-algebras. However, these maps cannot be lifted to $\ucIFrob_0$-algebra maps, not even up to homotopy, because the underlying objects $A$ and $B$ are obviously not isomorphic.
\end{ex}

Thus the functor \eqref{equ:cHLV functor} cannot be expected to be full. However, the following result shows that it is at least fully faithful onto the invertible morphisms.

\begin{prop}\label{prop:cmap}
    Let $(V,\alpha)$ and $(W,\beta)$ be two $\ucIFrob_{n,\infty}$-algebras and let $\phi:V\to W$ be an $\infty$-morphism of $\uIFrob_{n,\infty}$-algebras. Then the fiber of 
    \[
    \Map_{\cHLVCat}((V,\alpha), (W,\beta))
    \to 
    \Map_{\HLVCat}((V,\alpha), (W,\beta))
    \]
    over $\phi$
    is either empty or contractible. It is contractible iff the linear part $\phi_1$ of $\phi$ induces an isomorphism on cohomology.
\end{prop}

We may hence summarize the findings of this section as follows.

\begin{cor}\label{cor:uc equiv}
The forgetful functor $\cHLVCat \to \HLVCat$ identifies $\cHLVCat$ with the groupoid core of the full subcategory of non-degenerate unital involutive Frobenius algebras, that is, we obtain an equivalence of $\infty$-categories ($\infty$-groupoids)
\[
\cHLVCat \to \HLVCat^{\text{non-deg}, \simeq}.
\]
\end{cor}

\newcommand{\bba}{{\bullet\rightsquigarrow\bullet}}
\newcommand{\bbu}{{\bullet\to\bullet}}
\newcommand{\bbe}{{\bullet\,\bullet}}

\subsection{Proof of Proposition \ref{prop:cmap}}
For the proof we will use a different characterisation of the mapping spaces in $\HLVCat$ and $\cHLVCat$ due to \cite{HoffbeckLerayVallette}.
Let $\uIFrob_{n,\infty}^\bbe$ be the two-colored properad governing a pair of $\uIFrob_{n,\infty}$-algebras.
Let $\uIFrob_{n,\infty}^\bba$ be the two-colored properad governing a pair of $\uIFrob_{n,\infty}$-algebras and an $\infty$-morphism between them.
Similarly, define $\ucIFrob_{n,\infty}^\bbe$ and $\ucIFrob_{n,\infty}^\bba$.
Let $\End_{V,W}$ be the two colored endomorphism properad. (See \cite[Section 1]{HoffbeckLerayVallette} for more details.)
Then the mapping spaces we consider fit into the (horizontal) fiber sequences:
\[
\begin{tikzcd}
    F \ar{r}{\simeq} \ar{d}& F'\ar{r}\ar{d} & * \ar{d} \\
    \Map_{\cHLVCat}((V,\alpha), (W,\beta))
    \ar{r}\ar{d}
    & 
    \Map_{P}(\ucIFrob_{n,\infty}^\bba,\End_{V,W})
    \ar{d}\ar{r}
    & 
    \Map_{P}(\ucIFrob_{n,\infty}^\bbe,\End_{V,W})
    \ar{d}
    \\
    \Map_{\HLVCat}((V,\alpha), (W,\beta))
    \ar{r}
    &
    \Map_{P}(\uIFrob_{n,\infty}^\bba,\End_{V,W})
    \ar{r}
    & 
    \Map_{P}(\uIFrob_{n,\infty}^\bbe,\End_{V,W})
\end{tikzcd}.
\]
Our goal is to compute the vertical fiber $F$ and show that it is either empty or a contractible. The right-most vertical fiber is contractible by Proposition \ref{prop:ucIF uIF fiber}. Hence we may equivalently compute the middle vertical fiber $F'$, which is weakly equivalent to our target $F$.

To compute this middle fiber in turn we will use a specific model of the cofibration $\uIFrob_{n,\infty}^\bba\to \ucIFrob_{n,\infty}^\bba$, constructed akin to the morphism $\uIFrob_{n,\infty}\to X_{n,\infty}$ of Section \ref{sec:Xn} above.
To this end, consider the following diagram:
\begin{equation}\label{equ:bigsquig}
\begin{tikzcd}
    \FreeP(\Delta,u,f)\ar[hookrightarrow]{d} \ar{r}
    &
    \uIFrob_{n,\infty}^\bba 
    \ar{r}{\sim}
    \ar[hookrightarrow]{d}
    \ar[bend left]{dd}
    &
    \uIFrob_{n}^\bbu
    \ar{d}
    \ar[bend left]{dd}
    \\
    (\FreeP(\Delta,u,f,c',h), d) \ar{r}
    \ar{dr}
    &
    X_{n,\infty}^\bba 
    \ar{r}{\blacklozenge}
    \ar[dashed]{d}
    &
    X_{n}^\bbu
        \ar[dashed]{d}
    \\
    & \ucIFrob_{n,\infty}^\bba
    \ar{r}{\sim}
    &
    \ucIFrob_{n}^\bbu
\end{tikzcd}.
\end{equation}
Here the free (quasi-)objects on the left are 2-colored properads. The generators are as follows:
\begin{itemize}
    \item The generator $\Delta$ (the diagonal) is $(-1)^n$-symmetric of degree $n$ in arity $(2,0;0,0)$, i.e., it has two outputs of color 1.
    \item The generator $u$ (the unit) is of arity $(1,0;0,0)$ and degree 0.
    \item The degree zero generator $f$ (the morphism) is of arity $(0,1;1,0)$, i.e., it has one input of color 1 and an output of color 2.
    \item The element $c'$ is the counit in the second color of arity $(0,0;0,1)$ and degree $-n$.
    \item The homotopy $h$ has arity $(1,0;0,0)$ and degree -1. 
\end{itemize}
The differential is defined as 
\begin{align*}
    d\Delta &= du = df = dc' =0 \\
    dh &= (\mathit{id}\otimes c'f) \Delta - u.
\end{align*}
Now return to the diagram \eqref{equ:bigsquig}.
The object $\uIFrob_{n}^\bbu$ is the 2-colored properad governing a pair of $\uIFrob_{n}$-algebras and a strict morphism between them, and likewise for $\ucIFrob_{n}^\bbu$.
The objects $X_{n,\infty}^\bba$ and $X_{n}^\bbu$ are defined as pushouts in their respective squares. The dashed arrows are given by the universal property of the pushout.
The arrows from the quasi-freee objects are by sending the generators to the like-named generators of $\uIFrob_{n,\infty}^\bba$ and $\uIFrob_{n,\infty}^\bba$ respectively.

\begin{lemma}\label{lem:xni s qiso}
    The dashed arrows in \eqref{equ:bigsquig} are quasi-isomorphisms.
\end{lemma}
\begin{proof}
    By copying the proof of Lemma \ref{lem:Xni qiso} one readily sees that the arrow labeled "$\blacklozenge$" in \eqref{equ:bigsquig} is a quasi-isomorphism.
    Hence to show the lemma it suffices to show that the right-hand dashed arrow $X_n^{\bbu}\to \ucIFrob_n^{\bba}$ is a quasi-isomorphism.
    To this extend we will closely follow the proof of the analogous Lemma \ref{lem:Xn qiso} and compute the cohomology of $X_n^{\bbu}$.
    We have, assuming $k=0$ if $q>0$,
    \begin{align*}
    X_n^{\bbu}(k,l;p,q)
    &=
    \bigoplus_{i,j\geq 0 \atop {k+p+i\geq 1 \atop q>0\Rightarrow k=0} }
    X_n^{\bbu}(k,l+j;p+i,q)
    \otimes_{S_i\times S_j} \Q[n]^{\otimes i}\otimes \Q[1]^{\otimes j}
    \\
    &=
    \bigoplus_{i\geq 0 \atop k+p+i\geq 1 }
    \left(
    \underbrace{X_n^{\bbu}(k,l;p+i,q)
    \otimes_{S_i}\Q[n]^{\otimes i}
    }_{=\Q a_{k,l,p,q,i}}
    \oplus 
    \underbrace{X_n^{\bbu}(k,l+1;p+i,q)
    \otimes_{S_i} 
    \Q[n]^{\otimes i}\otimes \Q[1]
    }_{=\Q b_{k,l,p,q,i}}
    \right).
    \end{align*}
    If $k>0$ and $q>0$ then the complex is identically zero.
    The differential is, on the indicated basis elements 
    \begin{align*}
    d a_{k,l,p,q,i} &= 0 
    \\
    d b_{k,l,p,q,i} &= a_{k,l,p,q,i+1}-a_{k,l,p,q,i}. 
    \end{align*}
    We may compute the cohomology as in the proof of Lemma \ref{lem:Xn qiso} and find that $H(X_n^{\bbu}) \cong \ucIFrob_n^{\bba}$, with the cohomology represented by the elements $a_{k,l,p,q,0}$ (for $k+p\geq 1$) and $a_{0,l,0,q,1}$ (for $k=p=0$).
    The morphism $X_n^{\bbu}\to \ucIFrob_n^{\bbu}$ maps these generators to the generators of $\ucIFrob_n^{\bbu}$, so that the morphism is a quasi-isomorphism and the lemma is shown.
\end{proof}

The argument of the preceding proof also works if we replace the role of $\uIFrob_{n}^\bbu$ by $\ucIFrob_{n}^\bbu$ and analogously to Corollary \ref{cor:htpy epi} we obtain:

\begin{cor}\label{cor:htpy epi 2}
    All three squares in the diagram 
    \[
\begin{tikzcd}
    \FreeP(\Delta,u,f) \ar[d] \ar[r] &  \uIFrob_n^\bbu \ar[r] \ar[d] & \ucIFrob_n^\bbu \ar[d] \\
    (\FreeP(\Delta,u,f,c',h),d) \ar[r] & \ucIFrob_n^\bbu \ar[r] & \ucIFrob_n^\bbu
\end{tikzcd}
\]
are homotopy pushout squares in the category of two-colored dg properads.
In particular, the inclusion $\uIFrob_n^\bbu\to \ucIFrob_n^\bbu$ is a homotopy epimorphism.
\end{cor}

We continue with our proof of Proposition \ref{prop:cmap}. We first show that the fiber $F'$ is either contractible or empty. 
To this end we use that our mapping spaces are homotopy mapping spaces, so that the fiber is the same as the homotopy fiber of the map of derived mapping spaces $\Map^h(\ucIFrob_n^\bbu,\End_{V,W})\to \Map^h(\ucIFrob_n^\bbu,\End_{V,W})$, where the derived mapping spaces are taken in the category of 2-colored dg properads. But the map $\uIFrob_n^\bbu\to \ucIFrob_n^\bbu$ is a homotopy epimorphism by Corollary \ref{cor:htpy epi 2}, and applying $\Map^h(-,\End_{V,W})$ to a homotopy epimorphism produces a homotopy monomorphism (cf. the proof of Proposition \ref{prop:ucIF uIF fiber} above), hence we conclude that the fiber $F'$ is indeed either contractible or empty.

It remains to check the condition for emptyness of the fiber in Proposition \ref{prop:cmap}. From our left-hand pushout square in \eqref{equ:bigsquig} we obtain the pullback square
\begin{equation}\label{equ:bigsquig 2}
\begin{tikzcd}
    F'' \ar{r}
    \ar{d}{=} 
    & 
    \Map_P(X_{n,\infty}^\bba,\End_{V,W}) 
    \ar[twoheadrightarrow]{r}
    \ar{d}
    & 
    \Map_P(\uIFrob_{n,\infty}^\bba,\End_{V,W})\ar{d}
    \\
    F'' \ar{r} & 
    \Map_P((\FreeP(\Delta,u,f,c',h), d),\End_{V,W})
    \ar[twoheadrightarrow]{r}
    &
    \Map_P(\FreeP(\Delta,u,f),\End_{V,W})
\end{tikzcd}
\end{equation}
We have also indicated the fiber $F''$ on the left-hand side. It is weakly equivalent to the simplicial set $F'$ we desire to compute because of Lemma \ref{lem:xni s qiso}.

To compute the fiber $F''$ we use the bottom row in the diagram above.
The fiber is taken over the point determined by our given $\infty$-morphism of $\uIFrob_{n,\infty}$-algebras $\phi:V\to W$. Let $\phi_1$ be its linear part. 
Then in order for the fiber to be nonempty, we must be able to find an element $h_V\in V$ and a cocycle $c'_W\in W^*$ so that 
\[
d h_V = (\mathit{id}\otimes c_W'
\phi_1)(\Delta_V) - u_V,
\]
where $u_V\in V$ is the unit and $\Delta_V\in V\otimes V$ is the diagonal of $V$.
This is possible iff there is a cocycle $c_W'\in W^*$ such that in the cohomology $H(V)$ we have $[(\mathit{id}\otimes c_W'\phi_1)(\Delta_V)]=[u_V]$.
By Lemma \ref{lem:ifrob morph} this in turn is equivalent to the cohomology morphism $[\phi_1]$ being an isomorphism, i.e., to $\phi_1$ being a quasi-isomorphism.
 
\hfill\qed

\appendix

\section{More general picture}
The adjunction of Theorem \ref{thm:main 1} can also be understood as an instance of a symmetric monoidal Morita equivalence. We sketch a less explicit proof of Theorem  \ref{thm:main 1} in this appendix.
Given a symmetric monoidal dg-category $\cC$  and an object $x \in \cC$  we can consider the endomorphism prop $\POp(j,i) = \Hom(x^{\otimes i}, x^{\otimes j})$. It comes together with a symmetric monoidal adjunction
\begin{align*}
\cC \leftrightarrows& P\Mod \\
y \mapsto& \left( n \mapsto \Hom(x^{\otimes n}, y) \right).
\end{align*}
For example, the situation considered in this paper corresponds to the choices $\cC = \pois_d\Mod$ and $x \colon n \mapsto \ucom_d(n,1)\cong \Q^{\oplus n}$. Then 
$x^{\otimes i}= \ucom(i)$ (with $\otimes$ the point-wise monoidal product) and $\POp= \Hom_{\pois_d\Mod}( \ucom, \ucom )$.

However, the above construction of $\POp$ is not homotopically correct and needs to be altered, as we do not take the derived homomorphisms.
Using Proposition \ref{prop:derivedendoprop} below there is an appropriate homotopical way to perform this construction. One obtains a different and a priori unwieldy dg-prop $\POp$ such that
\begin{align*}
\POp \simeq& R\HHom_{\pois_d}( \ucom, \ucom ) = R\HHom_{\pois_d}( \pois_d \circ_{\lie_n} 1, \ucom) \\
&= R\HHom_{S}(1 \circ_{\lie_n} 1, \ucom) = \rcom_n \circ \ucom
\end{align*}
where we used Koszul duality for the $\Lie$ operad. A more careful analysis of this calculation also shows that $\cC \to \POp\Mod$ is an equivalence. Indeed, the Koszul duality for the $\Lie$ operad implies that the Koszul complex satisfies $\ucom \simeq \pois_d \circ_\tau \com^*$, from which it follows that $\ucom(i,-)$ are projective generators of $\pois_d\Mod$. Finally, one can show that $\POp$ is formal for $d \geq 2$ using a similar approach to \cite{HoreldeBritoCirici}. Namely, $\pois_d$ carries an extra grading (given by rescaling the Lie bracket)
which induces a grading on $P$. But now the above calculation shows that on the cohomology of $\POp(i,j)$ the extra grading coincides with the cohomological grading and thus $\POp$ is formal.

Note that by fibrantly/cofibrantly replacing $x^{\otimes n}$ we can define the dg category $\POp$ with hom spaces $\POp(i,j) = R\Hom(x^{\otimes n}, x^{\otimes m})$ and obtain an adjunction as above. This is generally not symmetric monoidal.
\begin{prop}\label{prop:derivedendoprop}
    Let $x \in \cC = A\Mod$, where $A$ is a (small) Hopf dg-category. Then the adjunction
    \[
    \cC \leftrightarrows \POp\Mod
    \]
    can be enhanced to a symmetric monoidal Quillen adjunction. In particular, $\POp$ is (equivalent as dg-category) to a dg-prop.
\end{prop}
\newcommand{\Bij}{\mathrm{Fin}^{\simeq}}
\begin{proof}[Proof (sketch)]
    We can assume that $x$ is fibrant and cofibrant and choose fibrant replacements $\nu_n \colon x^{\otimes n} \to x_n$. Let $\cO$ be the ($\N_0$-colored where we denote the colors by $c_n$) endomorphism operad and let $\alpha \colon \cF(\mu) \to \cO$ be the natural map from the free operad $\cF(\mu)$ with generators $\mu_n \colon c_1^{\otimes n} \to c_n$, $n=0,1,2,\dots$. Let $\cM(\Bij)$ be the localization of $\cF(\mu)$ at $\mu_n^{-1}$, $n=0,1,2,\dots$. This can be seen to be the $\N$-colored prop obtained from $\Bij$ by pulling back along the map on objects $\N[\N] \to \N$. Factor $\alpha \colon \cF(\mu) \to \widetilde{\Env \cO} \xrightarrow{\simeq} \Env\cO$ into a cofibration and an equivalence. We now define $\QOp$ to be the pushout
    \begin{equation}
        \label{diag: endo prop pushout}
    \begin{tikzcd}
        \cF(\mu) \ar[r, "\alpha"] \ar[d] & \widetilde{\Env(\cO)} \ar[d] \\
        \cM(\Bij) \ar[r]& \QOp.
    \end{tikzcd}
    \end{equation}.
    Let $\POp$ be the restriction of $\QOp$ to the color $c_1$, i.e., we keep only the operations with all inputs and outputs in color $c_1$. Note that $\POp$ and $\QOp$ are equivalent as categories (as opposed to merely weakly equivalent) since the map $\cM(\Bij) \to \QOp$ encodes isomorphisms $c_n \cong c_1^{\otimes n}$. We thus obtain symmetric monoidal left Quillen functors
    \[
    \cC \longleftarrow \widetilde{\Env(\cO)}\Mod \longrightarrow \POp\Mod.
    \]
    Define $\widetilde{P}\subset \widetilde{\Env\cO}$ be the full dg subcategory on the objects $c_i$, $i=0,1,2,\dots$, so that $\widetilde{P}(c_i,c_j) = \widetilde{\Env\cO}(c_i,c_j)$.
    
    {\bf Claim:} The induced map $\widetilde{P} \to \QOp$ is a weak-equivalence of dg-categories. 
    
    To see that we first note that the pushout \eqref{diag: endo prop pushout} is a localization and thus by Day's Theorem the underlying ($\N[\N]$-colored) category is again a pushout. Moreover, both $\cF(\mu) \to \cM(\Bij)$ and $\cF(\mu) \to \widetilde{\Env(\cO)}$ are homotopically flat in the sense of \cite{ChuangLazarev} and thus loc. cit. gives a filtration on $\cQ$ with associated graded
    \[
    \widetilde{\Env(\cO)}, \ \widetilde{\Env(\cO)} \circ_{\cF(\mu)} \overline{\cM{\Bij}}, \ \widetilde{\Env(\cO)} \circ_{\cF(\mu)} \overline{\cM{\Bij}} \circ_{\cF(\mu)} \overline{\widetilde{\Env(\cO)}}, \ \dots
    \]
    One readily verifies that for any $\cF(\mu)$ right module $X$ the map $X \to X \circ_{\cF(\mu)} \cM{\Bij}$ is given by the right action of $\mu_{i_1}\otimes \cdots \otimes \mu_{i_k}$
    \[
     X(c_{i_1} \otimes \dots \otimes c_{i_k}) \to X(c_1^{\otimes \sum i_j}).
    \]
    In particular, the map $\widetilde{\Env(\cO)}(c_{i_1} \otimes \dots \otimes c_{i_k},c_n) \to \left(\widetilde{\Env(\cO)} \circ_{\cF(\mu)} \cM{\Bij}\right)(c_{i_1} \otimes \dots \otimes c_{i_k},c_n)$ is homotopic to
    \[
    \Hom_{\cC}(x_{i_1} \otimes \dots \otimes x_{i_k},x_n) \to \Hom_{\cC}( \left(x_1\right)^{\otimes \sum i_j}, x_n)
    \]
    and thus an equivalence. We thus conclude that $\left(\widetilde{\Env(\cO)} \circ_{\cF(\mu)} \overline{\cM{\Bij}}\right)(-, c_n)$ is contractible and the claim follows.
    
    We have thus obtained the following diagram of (not necessarily symmetric monoidal) left Quillen functors,
    \[
    \begin{tikzcd}
        \POp\Mod \ar[d] & \ar[l, "\simeq"] \widetilde{P}\Mod \ar[d] \ar[rd, "\simeq"] & \\
        \cC & \ar[l] \widetilde{\Env(\cO)}\Mod \ar[r] & \POp\Mod.
    \end{tikzcd}
    \]
    From the right commuting triangle it follows that $\widetilde{\Env\cO}\Mod \to \POp\Mod$ exhibits $\POp\Mod$ as a symmetric monoidal left Bousfield localization of $\widetilde{\Env\cO}\Mod$. To conclude the proof, it remains to show that $\widetilde{\Env(\cO)}\Mod \to \cC$ factors through that localization. However, this follows from the same argument as above. Namely, let $\NOp$ be the $A\dash\widetilde{\Env(\cO)}$ bimodule inducing the functor $\widetilde{\Env(\cO)}\Mod \to \cC$, then we obtain that $\NOp \to \NOp \circ_{\widetilde{\Env(\cO)}} \POp$ is an equivalence using that (composition with) $\mu_n$ is an equivalence in $\cC$.





\end{proof}

\end{document}